\documentclass[10pt,letterpaper]{amsart}

\usepackage{color}

\usepackage{graphicx, epstopdf}
\usepackage{cases}
\usepackage{bbm}
\usepackage{amssymb}
\usepackage{txfonts}
\usepackage{amscd}
\usepackage{amsfonts,latexsym,amsmath,amsxtra,mathdots,amssymb,latexsym,mathtools}
\usepackage{mathabx}
\usepackage[all,cmtip]{xy}

\usepackage{colordvi}
\usepackage{multicol}
\usepackage{hyperref}
\usepackage{tikz-cd}
\usepackage{float}
\usepackage{setspace, floatflt}
\usepackage{tensor}
\usepackage[normalem]{ulem}
\usepackage{euscript, epstopdf}

\allowdisplaybreaks

 \newcommand{\BC}{{\mathbb {C}}} 
 \newcommand{\BG}{{\mathbb {G}}} \newcommand{\BH}{{\mathbb {H}}}
  
 \newcommand{\BN}{{\mathbb {N}}} 
 \newcommand{\BR}{{\mathbb {R}}} 
  
 \newcommand{\BZ}{{\mathbb {Z}}}

\newcommand{\CI}{{\mathcal {I}}} 
\newcommand{\CM}{{\mathcal {M}}} \newcommand{\CO}{{\mathcal {O}}}
 \newcommand{\CT}{{\mathcal {T}}}

\newcommand{\GL}{{\mathrm {GL}}} 
\newcommand{\SL}{{\mathrm {SL}}} 
\newcommand{\SO}{{\mathrm{SO}}}
 
\newcommand{\Sp}{{\mathrm{Sp}}}

\newcommand{\Hom}{\mathrm{Hom}}
\newcommand{\ind}{\mathrm{ind}} \newcommand{\Ind}{\mathrm{Ind}}

\newcommand{\ra}{\rightarrow}

\def\frm{\mathfrak{m}}
\def\frp{\mathfrak{p}}

\def\-{^{-1}}

\def\shskip{\hskip 0.5 pt}

\makeatletter
\g@addto@macro\normalsize{\setlength\abovedisplayskip{3pt}}
\makeatother

\makeatletter
\g@addto@macro\normalsize{\setlength\belowdisplayskip{3pt}}
\makeatother

\newcommand{\delete}[1]{}

\theoremstyle{plain}

\newtheorem{thm}{Theorem}[section] \newtheorem{cor}[thm]{Corollary}
\newtheorem{lem}[thm]{Lemma}  \newtheorem{prop}[thm]{Proposition}

\newtheorem {conj}[thm]{Conjecture} \newtheorem{defn}[thm]{Definition}

\newtheorem {rem}[thm]{Remark}

\numberwithin{equation}{section}

\newcommand{\Ext}[1]{\textup{Ext}^{#1}}
\newcommand{\St}{\textup{St}}\newcommand{\EP}{\textup{EP}}
\newcommand{\transpose}[1]{{}^t #1}\newcommand{\transinv}[1]{{}^t#1^{-1}}

\begin{document}

	\title{Ext-distinction for $p$-adic symmetric spaces}

	\author{Chang Yang}

	\address{Key Laboratory of High Performance Computing and Stochastic Information Processing (HPCSIP), Hunan Normal University, School of Mathematics and Statistics, Changsha, 410081, China}
	\email{cyang@hunnu.edu.cn}
	
	
\maketitle

\begin{abstract}
	 Let $G/H$ be a $p$-adic symmetric space. We compute explicitly the higher relative extension groups for all discrete series representations of $G$ in two examples: the symplectic case and the linear case. The results have immediate applications to the computation of the Euler-Poincar\'e pairing, the alternating sum of the dimensions of the Ext-groups. In the linear case we confirm a conjecture of Wan which asserts the equality of the Euler-Poincar\'e pairing with the geometric multiplicity for any irreducible representation of $G$. In the symplectic case we reduce the verification of Wan's conjecture to the case of discrete series representations. We also determine all relatively supercuspidal representations in both cases.
\end{abstract}

\section{Introduction}

\subsection{Main results}

Throughout this work, let $F$ be a nonarchimedean local field of characteristic $0$. Let $G = \BG(F)$ be be the group of $F$-points of a connected reductive group $\BG$ over $F$ and $H = G^{\theta}$ be the subgroup of fixed points of $\theta$ of $G$, where $\theta$ is an involution on $\BG$ defined over $F$. Let $A_{G,H}$ be the maximal split torus of $Z_G \cap H$ with $Z_G$ the center of $G$. For an irreducible representation $\pi$ of $G$, the multiplicity $m(\pi)$ of $\pi$ is defined by
\begin{align*}
	m(\pi) = \dim \Hom_H (\pi, \BC).
\end{align*}
Harmonic analysis on the symmetric space $G/H$ requires a good understanding of $m(\pi)$ (c.f. \cite{SV-periods}). In \cite{Wan-Multiplicity-JEMS}, build upon the works where a relative trace formula is exploited to study the multiplicities, Wan defines a geometric multiplicity $m_{geom}(\vartheta)$ for any quasi-character $\vartheta$ on $G$ as an integral formula involving $\vartheta$ ( the definition in loc. cit. is for general spherical varieties, and we refer to \cite[Definition 6.1]{Wan-Multiplicity-JEMS} for a precise form of the definition.) The conjectured multiplicity formula in \cite[Conjecture 6.4]{Wan-Multiplicity-JEMS} is that
\begin{align}\label{eq::Intro-Mul=geomMul}
	m(\pi) = m_{geom}(\theta_{\pi}),
\end{align}
where $\theta_{\pi}$ is the Harish-Chandra character of $\pi$, for any supercuspidal representation $\pi$ of $G$. Moreover, when $G/H$ is tempered (resp. strongly tempered), that is, all the matrix coefficients of discrete series representations (resp. tempered representations) of $G$ are integrable over $H/A_{G,H}$, the multiplicity formula \eqref{eq::Intro-Mul=geomMul} holds for all discrete series representations (resp. tempered representations). We refer the reader to \cite{Wan-Multiplicity-JEMS} and references therein for the examples where the multiplicity formula has been proven.

In \cite{Prasad-Ext-BranchingLaw-ICM2018}, Prasad proposes a homological algebra approach to the study of $m(\pi)$ in the context of branching laws for representations of a group to its subgroups. In our context, according to \cite{Prasad-Ext-BranchingLaw-ICM2018}, it is natural to consider the higher extension groups $\Ext{n} (\pi|_H,\BC)$ in the category of smooth representations of $H/A_{G,H}$. As observed by Prasad, in order for \eqref{eq::Intro-Mul=geomMul} to hold for all irreducible representations, the multiplicity $m(\pi)$ should be replaced by $\EP(\pi)$, the Euler-Poincar\'e paring of $\pi$, the alternating sum of the dimensions of the $\textup{Ext}$-groups. The conjectured $\textup{EP}$ formula in \cite[Conjecture 6.5]{Wan-Multiplicity-JEMS} is then
\begin{align}\label{eq::Conjecture-EP}
	\EP(\pi)  = m_{geom}(\theta_{\pi}).
\end{align}
for any finite length representation $\pi$ of $G$. In view of the multiplicity formula \eqref{eq::Intro-Mul=geomMul}, it is natural to expect the vanishing of higher extension groups:
\begin{align*}
	\Ext{n}(\pi|_H,\BC) = 0,\quad n >  0
\end{align*}
for every discrete series representation (resp. tempered representations) $\pi$ when $G/H$ is tempered (resp. strongly tempered). These expectations reflect  a deep relationship between harmonic analysis on $G/H$ and homological algebra of $G/H$.

The main purpose of this paper is to verify the $\EP$ formula \eqref{eq::Conjecture-EP} in two specific examples, the symplectic case and the linear case. Since both sides of \eqref{eq::Conjecture-EP} can be viewed as a function on the Grothendieck group of the category of smooth representations of finite length of $G$, it suffices to verify \eqref{eq::Conjecture-EP} for standard modules of $G$. By inductive arguments for $m_{geom}(\theta_{\pi})$ and $\EP_{H}(\pi,\BC)$ respectively in Section \ref{sec::EP-symplectic} and Section \ref{sec::EP-linear}, we reduce the verification to the case of discrete series. 

\subsubsection{The linear case}

Let $G= \GL_{2n}(F)$, $Z$ be the center of $G$ and $H = \GL_{n}(F) \times \GL_{n}(F)$. The multiplicity formula for discrete series has been established by Beuzart-Plessis and Wan in \cite{Beuzart-Plessis-Wan-Duke-Shalika}, combining with the relationship with Shalika model. Our main result in this case is the following vanishing result, from which \eqref{eq::Conjecture-EP} follows.

\begin{thm} 
	Let $\pi$ be a discrete series representation of $G$ with trivial central character. Then
	\begin{align*}
		\Ext{n}_{H/Z} (\pi,\BC) = 0,\quad n >0.
	\end{align*}
\end{thm}

\begin{rem}
	In the Jacquet-Rallis case $(G,H) = (\GL_{n}(E), \GL_{n}(F))$ where $E/F$ is a quadratic extension, we have a similar vanishing result. The proof is almost identical to that in the linear case. So in the Jacquet-Rallis case, the $\EP$ formula \eqref{eq::Conjecture-EP} also holds.
\end{rem}
As $G/H$ is a tempered variety, our result meets the general expectation mentioned above. In other contexts, there are some recent progresses on the $\textup{Ext}$-vanishing results, see \cite{Chan-Savin-Ext-Vanishing-Duke,Chan-Ext-StandardModule-MathZ} by Chan for the Rankin-Selberg case,  \cite{Cai-Fan-TripleProduct-MRL} by Cai and Fan for the triple product case, and \cite{Chen-Ext-Vanishing-Bessel} by Chen for the Gan-Gross-Prasad case.

\subsubsection{The symplectic case}

Let $G= \GL_{2n}(F)$ and $H = \Sp_{2n}(F)$. The multiplicity formula for cuspidal representations has been proven by the recent work of Beuzart-Plessis and Wan \cite{Beuzart-Plessis-Wan}. Our main result in this case is the following computations.
\begin{thm}
	Let $\rho$ be a cuspidal representation of $\GL_d(F)$ and $\pi = \St_k(\rho)$ be the generalized Steinberg representation of $G$ ($2n = kd$). Then
	\begin{itemize}
		\item If $k \neq 2$, we have
		\begin{align*}
			\Ext{i}_H (\St_k (\rho), \BC) = 0\ \textup{ for all }i \geqslant 0.
		\end{align*}
		\item 
		If $k = 2$, we have
		\begin{align*}
			\Ext{1}_H (\St_2(\rho), \BC) = \BC, \quad \Ext{i}_H (\St_2(\rho), \BC) = 0,\ i \neq 1.
		\end{align*}
	\end{itemize}	
\end{thm}

Due to the lack of precise information of characters of discrete series representations, we are unable to very \eqref{eq::Conjecture-EP} except for the $n = 1$ case, which is a direct consequence of Kazhdan orthogonality (see \cite[Section 6]{Prasad-IHES-Note}). Note that in this case $G/H$ is not tempered, so our computations are in alignment with the general expectations mentioned above.

In this paper we also discuss another notion that is important for the harmonic analysis on $p$-adic symmetric spaces, \emph{the relatively supercuspidal representations}. By the subrepresentation theorem of Kato-Takano \cite{Kato-Takano-Subrepresentation-symmetric}, these representations are the building blocks of irreducible distinguished representations (i.e., $m(\pi) \neq 0$), the same role as supercuspidal representations to irreducible representations. In \cite{Cai-Fan-RelativeCuspidal-Ext}, Cai and Fan introduced a new idea to determine relatively supercuspidals by computing the highest extension groups via the Schneider-Stuhler duaility. Following their idea, we determine all relatively supercuspidals in both cases in Proposition \ref{prop::RSC-symplectic} and Proposition \ref{prop::RSC-linear}. We refer the reader to Section \ref{sec::GLn} for unexplained notation below.

\begin{thm}\label{thm::Intro-RSC}
   $\ $

  \textup{(1)}$\quad$In the symplectic case, an irreducible representation $\pi$ of $G$ is relatively supercuspidal if and only if $\pi = Z(\nu^{-1/2}\rho,\nu^{1/2}\rho)$ is the unique subrepresentation of $\nu^{-1/2}\rho \times \nu^{1/2}\rho$  for some cuspidal $\rho$.
  
  \textup{(2)} $\quad$ In the linear case, an irreducible representation $\pi$ of $G$ is relatively supercuspidal if and only if $\pi = I_P^G (\sigma)$ where $P = MU$ is the standard parabolic subgroup of type $(2n_1,2n_2,\cdots,2n_k)$ and $\sigma$ is a regular $M \cap H$-distinguished cuspidal representation of $M$.
\end{thm}
In the symplectic case, the representations in Theorem \ref{thm::Intro-RSC} have been shown to be relatively supercuspidal by Kato-Takano in \cite[Section 8.3]{Kato-Takano-Subrepresentation-symmetric} by the Jacquet module methods.


The paper is organized as follows. In Section \ref{sec::Preliminaries} we present some standard facts on the computations of extension groups. In particular we recall some facts on the Hochschild-Serre spectral sequences that is important for our computations later. In Section \ref{sec::gemeotric lemma} we give a formulation of the application of  the geometric lemma of Bernstein-Zelevinsky to $\textup{Ext}$ groups, following the framework of Offen in \cite{Offen-ParabolicInduction-JNT}. In Section \ref{sec::GLn} we recall some standard notions and well known facts about representations of general linear groups. In Section \ref{sec::symplectic} and Section \ref{sec::linear} we do explicit computations in two cases.

\subsection{Acknowledgement.} We warmly thank Chen Wan for his help in the writing of Section \ref{sec::EP-symplectic} and Section \ref{sec::EP-linear} and Dipendra Prasad for useful correspondences. We also thank Bin Xu, Kei Yuen Chan, Yangyu Fan and Rui Chen for useful conversations. Part of the work was done when the author was attending the BIRS-IASM workshop ``Arthur Packets'', and we want to thank the organizers for their warm hospitality. This work is supported by the National Natural Science Foundation of China (No. 12001191).

\section{Preliminaries}\label{sec::Preliminaries}

In this section we prepare some general facts that are useful for our computations later. Let $G$ be an $\ell$-group (c.f. \cite[Section 1]{B-Z-I}). Let $Z$ be a closed subgroup contained in the center $Z(G)$ of $G$ and $\chi$ be a character of $Z$. We denote by $\CM^{\,\chi}(G;Z)$ the category of smooth representations of $G$ on which $Z$ acts through the character $\chi$. When $Z = \{e\}$, we simply denote it by $\CM(G)$. We have a natural equivalence of categories between $\CM^{\mathbf{1}} (G,Z)$ and $\CM (G/Z)$, where $\mathbf{1}$ denotes the trivial representation of $Z$. The following fact lies at the basis for all our discussions below.
\begin{lem}
	The category $\CM^{\,\chi} (G,Z)$ is abelian and has enough projective objects and injective objects.
\end{lem}
\begin{proof}
	The category $\CM(G)$ has enough projective objects and injective objects (see \cite[Section 2]{Prasad-Ext-BranchingLaw-ICM2018}). Taking $(Z,\chi)$-coinvariants in projective objects and $(Z,\chi)$-invariants in injective objects in $\CM(G)$ will give rise to projective and injective objects in $\CM^{\,\chi}(G,Z)$.
\end{proof}

By the lemma above, we can then define the extension group $\Ext{q}_{\CM^{\,\chi} (G,Z)} (V,W)$ for any two smooth representations of $G$ on which $Z$ acts via $\chi$. When $Z = \{e\}$, we write simple $\Ext{q}_G(\cdot,\cdot)$ for $\Ext{q}_{\CM(G)} (\cdot ,\cdot)$.

Following \cite[Section 1.8]{B-Z-I}, let $M$ and $U$ be closed subgroups of $G$ such that $M$ normalizes $U$, $M \cap U =\{e\}$ and the subgroup $P = MU$ is closed in $G$. Let $\psi$ be a character of $U$ normalized by $M$. We use the same notations $I_{U,\psi}$, $i_{U,\psi}$ and $r_{U,\psi}$ for the functors as in \cite[Section 1.8]{B-Z-I}. Assume that $Z \subset M$ and that $U$ is a limit by compact subgroups so that $r_{U,\psi}$ is exact. For any $\ell$-group $Q$, denote by $\delta_Q$ the modulus character of $Q$. 

The Frobenius reciprocity takes the following form.
\begin{lem}\label{lem::Frobenius}
	Let $V_1$ resp. $V_2$ be smooth representations of $G$ on which $Z$ acts via $\chi$ resp. $\chi^{-1}$. Let $W$ be a smooth representation of $M$ on which $Z$ acts via $\chi$
	We have, for all $q \geqslant 0$,
	\begin{align*}
		\Ext{q}_{\CM^{\,\chi}(G,Z)}( V_1,I_{U,\psi}(W)) &\cong \Ext{q}_{\CM^{\,\chi}(M,Z)} (r_{U,\psi} (V_1), W) \\
		\Ext{q}_{\CM^{\,\chi}(G,Z)}( i_{U,\psi} (W),V_2^{\vee}) &\cong \Ext{q}_{\CM^{\,\chi}(M,Z)} ( W, \delta_M\delta_G^{-1} (r_{U,\psi^{-1}} (V_2))^{\vee}),
	\end{align*}
	where the superscript ${}^{\vee}$ stands for the smooth dual.
\end{lem}
\begin{proof}
	Both isomorphisms follow directly from \cite[Proposition 2.2]{Prasad-Ext-BranchingLaw-ICM2018} and \cite[Proposition 1.9]{B-Z-I}.
\end{proof}

If $U = \{e\}$, in this work, we will rather write $\ind_M^G$ and $\Ind_M^G$ for the ordinary induction $i_{U,\psi}$ and $I_{U,k\psi}$. The following result is an analogue of \cite[Proposition 2.6]{Prasad-Ext-BranchingLaw-ICM2018}. We thank Dipendra Prasad for providing us with a proof.
\begin{lem}\label{lem::tensor-hom-ext}
	Let $V_1$ resp. $V_2$ be smooth representations of $G$ on which $Z$ acts via $\chi$ resp. $\chi^{-1}$. Let $V$ be a smooth representation of $G$ on which $Z$ acts trivially. Then we have
	\begin{align*}
		\Ext{q}_{G/Z} (V_1 \otimes V_2, V^{\vee}) \cong \Ext{q}_{\CM^{\,\chi}(G,Z)} (V_1 \otimes V, V_2^{\vee}) \cong \Ext{q}_{\CM^{\,\chi^{-1}} (G,Z)} (V \otimes V_2, V_1^{\vee}) .
	\end{align*}
\end{lem}
\begin{proof}
	By \cite[Proposition 2.2]{Prasad-Ext-BranchingLaw-ICM2018}, we have
	\begin{align*}
		\Ext{q}_{G/Z} (V_1 \otimes V_2, V^{\vee}) \cong \Ext{q}_{G/Z} (V, (V_1 \otimes V_2)^{\vee}). 
	\end{align*}
	We have a canonical isomorphism 
	\begin{align*}
		(V_1 \otimes V_2)^{\vee} \cong \Hom_{\BC} (V_1,V_2^{\vee})^{\infty},
	\end{align*}
	where the superscript ${}^{\infty}$ stands for the subspace of smooth vectors in a representation. Thus it suffices for us to show that
	\begin{align}\label{formula::tensor-hom-ext}
		\Ext{q}_{G/Z} (V, \Hom_{\BC} (V_1,V_2^{\vee})^{\infty})  \cong \Ext{q}_{\CM^{\,\chi}(G,Z)} (V_1 \otimes V, V_2^{\vee}).
	\end{align}
	Note that $\bullet \mapsto \bullet \otimes V_1$ gives an exact functor from $\CM^{\mathbf{1}} (G,Z)$ to $\CM^{\,\chi} (G,Z)$. Note also that the functor $\bullet \mapsto \Hom_{\BC} (V_1, \bullet)^{\infty}$ is exact from $\CM^{\,\chi} (G,Z)$ to $\CM^{\mathbf{1}} (G,Z)$. This follows from the fact that, for any compact open subgroup $K$ of $G$, every smooth representation of $G$ is $K$-semisimple. The required isomorphism \eqref{formula::tensor-hom-ext} follows directly from \cite[Proposition 2.2]{Prasad-Ext-BranchingLaw-ICM2018}.
\end{proof}

We will need the following vanishing result for our computations.
\begin{lem}\label{lem::vanishing}
	Let $V$ and $W$ be smooth representations of $G$. Suppose that there exist an element $z \in Z(G)$ in the center of $G$ and two complex numbers $c_1$ and $c_2 $ with $c_1 \neq c_2$ such that $z \cdot v = c_1 v$ for all $v \in V$ and $z \cdot w = c_2 w$ for all $w \in W$. Then we have
	\begin{align*}
		\Ext{\sharp}_G (V,W) = 0.
	\end{align*}
\end{lem}
\begin{proof}
	The argument is essentially the same as that in the proof of \cite[Theorem 4.1]{Borel-Wallach}, which uses the intepretation of Ext-groups as exact sequences. So we omit the details (see also \cite[Lemma 13]{Orlik-Extensions-Steinberg-JoA}).
\end{proof}

We will also need a variant of the Hochschild-Serre spectral sequence. Recall first that the $i$-th homology $H_i(G,V)$ of a smooth $G$-module $V$ is defined to be the $i$-th left derived functor of the right exact functor of taking $G$-coinvariants:
\begin{align*}
	\CM (G) &\ra   \textup{Vec} \\
	V &\mapsto V_G.
\end{align*}
It is easily seen that 
\begin{align}\label{formula::dual-homology-ext}
	H_i (G,V)^{\ast}  \cong \Ext{i}_G (V, \BC),
\end{align}
for all smooth $G$-module $V$ and all $i \geqslant 0$. Here the superscript ${}^{\ast}$ stands for the dual vector space. The following result can be found in \cite[Lemma 5]{Orlik-Extensions-Steinberg-JoA} (see also \cite{Casselman-non-unitary-argument-homology}).
\begin{lem}\label{lem::HS-spectral}
	Let $N \subset G$ be a closed normal subgroup of $G$. If $V$ is a projective $G$-module, then $V_N$ is a projective $G/N$-module. Thus, for every pair of smooth $G$-modules $V$ and $W$ with $N$ acting trivially on $W$, we have a spectral sequence
	\begin{align*}
		E_2^{p,q} = \Ext{q}_{G/N} (H_p (N,V),W) \Rightarrow \Ext{p+q}_G (V,W).
	\end{align*}
\end{lem}

Note that the homology group $H_p(N,V)$, a priori a vector space over $\BC$, can be viewed canonically as a smooth $G/N$-module. In fact, take a projective resolution $P_{\bullet}$ of $V$ in $\CM(G)$. The restriction of $P_{\bullet}$ to $N$ are still projective in $\CM(N)$. The $N$-coinvariants of $P_{\bullet}$ are smooth as $G/N$-modules, so as the homology groups.

\begin{lem}\label{lem::homology-G1G2}
	Let $V_1$ and $V_2$ be smooth modules of $G_1$ and $G_2$ respectively. Then we have a canonical isomorphism of smooth $G_2$-modules
	\begin{align*}
		H_p (G_1, V_1 \boxtimes V_2 )  \cong H_p (G_1,V_1) \otimes V_2,
	\end{align*}
	where $G_2$ acts trivially on $H_p(G_1,V_1)$.
\end{lem}
\begin{proof}
	Let $P_{\bullet}$ be a projective resolution of $V_1$ in $\CM(G)$. We may assume that $G_1$ acts trivially on $V_2$ so that we write $V_1 \otimes V_2$ for $V_1 \boxtimes V_2$. Note that $P_i \otimes V_2$ is projective in $\CM(G)$ since it is a direct sum of $P_i$. Thus $P_{\bullet} \otimes V_2 $ is a projective resolution of $V_1 \otimes V_2$ in $\CM(G)$. The lemma follows easily from the canonical isomorphism $P_N \otimes V_2 \cong (P \otimes V_2)_N$. 
\end{proof}

For applications of the Hochschild-Serre spectral sequence in the sequel, the closed normal subgroup $N$ is often taken to be a subgroup of the central torus of $G$. Assume that $N \cong  (F^{\times})^n $ and that $V$ is a representation of $G$ on which $N$ acts trivially. By Lemma \ref{lem::homology-G1G2}, we have an isomorphism of smooth $G$-modules
\begin{align*}
	H_p((F^{\times})^n, V) \cong H_p ((F^{\times})^n, \BC)) \otimes V.
\end{align*}
We have a simple description of $H_p ((F^{\times})^n ,\BC)$ from the dual side in view of \eqref{formula::dual-homology-ext}. The following is a specialization of \cite[Corollary 2]{Orlik-Extensions-Steinberg-JoA}.
\begin{lem}\label{lem::Ext-torus}
	We have
	\begin{align*}
		\Ext{p}_{(F^{\times})^n} (\BC,\BC)  = 
		\begin{cases}
			\BC^{n \choose p},  & p = 0,1,\cdots,n,  \\
			0, &\text{otherwise.}
		\end{cases}
	\end{align*}
\end{lem}
For $p = 1$, we would rather like to provide the following simple computation for the first extension groups. 
\begin{lem}\label{lem::ext-1-torus}
	Let $V$ and $W$ be smooth $(F^{\times})^n$-modules on which $(F^{\times})^n$ acts trivially. Then we have a canonical isomorphism of vector spaces
	\begin{align*} 
		\Ext{1}_{(F^{\times})^n} (V,W) \cong \Hom_{\BC}^n (V,W).
	\end{align*} 
\end{lem}
\begin{proof}
	We use the interpretation of $\Ext{1}_{(F^{\times})^n} (V,W)$ in terms of exact sequences. There is a canonical isomorphism between the abelian group $\Ext{1}_{(F^{\times})^n} (V,W)$ and the abelian group of equivalent classes of short exact sequences
	\begin{align*}
		0 \ra W \ra E \ra V \ra 0,
	\end{align*}
	where the addition on the latter group is given by Baer sum (see \cite[Chapter III]{Maclane-Homology} for more details). Note that every extension $E$ of $V$ by $W$ is isomorphic to $W \oplus V$ as a vector space over $\BC$, so we need only to specify the $(F^{\times})^n$ action on $W \oplus V$. Given $L = (\ell_1,\cdots,\ell_n) \in \Hom_{\BC}^n (V,W)$, we may construct an extension $(W \oplus V,\mu_L)$ of $V$ by $W$ as follows:
	\begin{align*}
		0 \ra W \ra W \oplus V \ra V \ra 0,
	\end{align*}
	where 
	\begin{align*}
		\mu_L ((z_1,\cdots, z_n)) ((w,v))  =  (w + \sum \nu_F (z_i)\ell_i (v), v). 
	\end{align*}
	It is easy to check that the map $L \mapsto (W \oplus V,\mu_L)$ is an injective homomorphism of abelian groups. Now let $(W \oplus V,\mu)$ be any extension of $V$ by $W$, we show that $\mu = \mu_L$ for some $L \in \Hom_{\BC}^n(V,W)$. Since the action of $(F^{\times})^n$ on $V$ is trivial, to $\mu$ we can associate a map $\mu'$ defined by
	\begin{align*}
		\mu' :   (F^{\times})^n \times V &\ra W   \\
		\  (\underline{z},v) &\mapsto \mu(\underline{z}) v - v.
	\end{align*}
	Note that for any $v$, we have $\mu'(\underline{z}_1 \underline{z}_2, v) = \mu'(\underline{z}_1,v) + \mu'(\underline{z}_2,v)$ as $(F^{\times})^n$ acts trivially on $W$. By the smoothness of $\mu$, the image $\mu'((\mathfrak{o}_F^{\times})^n,v)$ in $W$ is a abelian subgroup of $W$ with finitely many values, hence must be $0$. Thus we can define $\ell_i \in \Hom_{\BC}(V,W)$ by
	\begin{align*}
		\ell_i (v) = \mu'((1,\cdots,1,\varpi_F,1,\cdots,1) ,v)
	\end{align*}
	for $i =1,\cdots,n$, where $\varpi_F$ is at the $i$th place. Let $L = (\ell_1,\cdots,  \ell_n)$. It is easy to verify that $\mu = \mu_L$.
\end{proof}

For two smooth representations $V$ and $W$ of $G$ on which $Z$ acts through $\chi$, we define the Euler-Poincar\'e paring
\begin{align*}
	\EP_{\CM^{\,\chi}(G,Z)} (V,W) = \sum_i (-1)^i \dim_{\BC} (\Ext{i}_{\CM^{\,\chi}(G,Z)} (V,W)),
\end{align*}
once it is well-defined, that is, $\Ext{i}_{\CM^{\,\chi}(G,Z)} (V,W)$ is finite dimensional for all $i \geqslant 0$. As before, we write simply $\EP_G$ for $\EP_{\CM(G)}$. We have the following vanishing result for $\EP$.

\begin{prop}\label{prop::EP-trivial}
	Let $ H \subset G$ be a closed subgroup of $G$ and  $\chi:H  \ra \BC$ be a  character of $H$. Let $V$ be a finite length smooth representation of $G$. If there is a subgroup $Z \subset Z(G) \cap H$ such that $Z \cong F^{\times}$, then
	\begin{align*}
		\EP_H (V,\chi) = 0,
	\end{align*}
	whenever it is well-defined.
\end{prop}
\begin{proof}
	By the additivity of $\EP$, we may assume that $\pi$ is an irreducible representation of $G$ so that it has a central character $\chi_{\pi}$. By Lemma \ref{lem::vanishing}, if $\chi_{\pi} \neq \chi$ on $Z$, then $\EP_H(\pi,\chi) =0$. If $\chi_{\pi} = \chi$ on $Z$, then we may assume that $V$ is a smooth representation of $H$ on which $Z$ acts trivially and $\chi = \mathbf{1}$ is the trivial character. By the Hochschild-Serre spectral sequence in Lemma \ref{lem::HS-spectral}, we have
	\begin{align*}
		E_2^{p,q} = \Ext{q}_{H/Z} (H_p (Z,V),\BC) \Longrightarrow \Ext{p+q}_H (V,\BC).
	\end{align*}
	By Lemma \ref{lem::Ext-torus}, $H_0(Z,V) = H_1(Z,V) = V$ and $H_i(Z,V) = 0$ for $i \geqslant 2$. Thus the spectral sequence collapse at the second page and $E_2^{p,q}$ is finite dimensional for all $p,q$. The proposition then follows from the Euler-Poincar\'e principle, which asserts that 
	$$\EP_H(V,\BC) = \sum_{p,q} (-1)^{p+q} \dim_{\BC} (E_2^{p,q}).$$
\end{proof}
\begin{rem}
	This proposition is a generalization of \cite[Lemma 4.1]{Prasad-Ext-BranchingLaw-ICM2018}, which, by Lemma \ref{lem::tensor-hom-ext}, can be viewed as the group case with $G= \GL_n(F) \times \GL_n(F)$ and $H = \Delta \GL_n(F)$ diagonally embedded in $G$
\end{rem}

We record the Kunneth theorm from \cite[Theorem 3.1]{Prasad-Ext-BranchingLaw-ICM2018}.
\begin{thm}\label{thm::Kunneth}
	Let $G_1$ and $G_2$ be two $\ell$-groups. Let   $E_1,F_1$ be any two smooth representations of $G_1$, and $E_2,F_2$ be any two smooth representations of $G_2$. Then assuming that $G_1$ is a reductive $p$-adic group, and $E_1$ has finite length, we have
	\begin{align*}
		\Ext{q}_{G_1 \times G_2} (E_1 \boxtimes E_2, F_1 \boxtimes F_2) \cong \mathop{\oplus}_{q = i + j } \Ext{i}_{G_1} (E_1 ,F_1) \otimes \Ext{j}_{G_2} (E_2, F_2).
	\end{align*} 
\end{thm}

\section{The geometric lemma}\label{sec::gemeotric lemma}

In this section we apply the geometric lemma of Bernstein and Zelevinsky \cite{B-Z-I} to Ext-distinction problems. We will work in the setting of $p$-adic symmetric spaces, following the framework given by Offen in \cite{Offen-ParabolicInduction-JNT}. Let $ \mathbf{G}$ be a connected reductive group over $F$ and $\theta$ be an involution on $\mathbf{G}$ defined over $F$. Let $G = \mathbf{G}(F)$ be the group of $F$-rational points of $\mathbf{G}$ and $H = \mathbf{H}(F)$ with $\mathbf{H} = \mathbf{G}^{\theta}$ the fixed point group of $\theta$ in $\mathbf{G}$. Let $ A_{G,H}$ be the maximal $F$-split torus in $ H \cap Z_G$ where $Z_G$ is the center of $G$. 

To the symmetric pair $(G,\theta)$ we associate the symmetric space
\begin{align*}
	X = \{ g \in G \mid \theta(g) = g^{-1}  \}
\end{align*}
equipped with the twisted $G$-action given by $g \cdot x = gx\theta(g)^{-1}$, $g \in G$, $x \in X $. For $x \in X$ we denote by $\theta_x$ the involution on $G$ given by 
\begin{align*}
	\theta_x (g) = x \theta(g) x^{-1}.
\end{align*}
For any subgroup $Q$ of $G$ and $x \in X$, let $Q_x$ be the stabilizer of $x$ in $Q$. 

 Fix a minimal parabolic subgroup $P_0$ of $G$ and a $\theta$-stable maximal split torus $T$ of $G$ contained in $P_0$, whose existence is proved in \cite[Lemma 2.4]{Helminck-Wang-Involution}. We assume for simplicity that $\theta (P_0) = P_0$. We call a parabolic subgroup (resp. Levi subgroup) of $G$ standard if it contains $P_0$ (resp. $T$).

  Let $P = MU$ be the Levi decomposition of a standard parabolic subgroup of $G$ and $\sigma$ be a smooth representation of $M$. Let $\pi = I_P^G (\sigma)$ be the normalized parabolic induction as defined in \cite[Section 2.3]{B-Z-I}. By \cite[Theorem 5.2]{B-Z-I} or \cite[Proposition 1.17]{Blanc-Delorme}, there is a filtration of the restriction to $H$ of $\pi$ such that the successive quotients are indexed by the $P$-orbits in $G \cdot  e \subset X$. For a $P$-orbit $\CO$, we denote by $\pi_{\CO}$ the composition factor corresponding to $\CO$. We then have
\begin{align}\label{formula::composition factor}
	\pi_{\CO} \cong \ind_{H^P_{\eta}}^H ((\sigma \delta_P^{1/2}|_{P_x})^{\eta}),
\end{align}
where $x$ is a representative in $\CO$, $x = \eta \cdot e$ for some $\eta \in G$ and  
\begin{align*}
	H^P_{\eta} = \eta^{-1} P \eta \cap H = \eta^{-1} P_x \eta.
\end{align*}
Here $(\sigma \delta_P^{1/2}|_{P_x})^{\eta} $ is the representation of $H^P_{\eta}$ obtained from $\sigma \delta_P^{1/2}|_{P_x}$ by conjugation by $\eta$.

Let $W$ be the Weyl group of $G$ with respect to $T$. As $\theta$ stabilize $T$, $\theta$ also acts as an involution on $W$. Let
\begin{align*}
	\mathfrak{I}_0(\theta) = \{ w \in W \colon w \theta(w) = e \}
\end{align*}
be the set of twisted involutions in $W$. Following \cite[Section 3]{Offen-ParabolicInduction-JNT} we can fiber the $P$-orbits in $X$ over certain twisted involutions. For a standard Levi subgroup $M$, let $W_M$ be the Weyl group of $M$ with respect to $T$. For two standard Levi subgroups $M$ and $M'$, let ${}_MW_{M'}$ be the set of all $w \in W$ that are left $W_M$-reduced and right $W_{M'}$-reduced. The fibration map 
\[
\iota_M: P\backslash X \rightarrow  ({}_MW_{\theta(M)} \cap \mathfrak{I}_0(\theta))
\]
is characterized by the identity
\begin{align*}
	P x \theta(P) = P\iota_M(P \cdot x) \theta(P).
\end{align*}
For a $P$-orbit $\CO$ in $X$, let $w = \iota_M (\CO)$. We have that $L = M \cap w \theta(M) w^{-1}$ is a standard Levi subgroup of $M$. 
\begin{defn}
	We say that $x \in X$ or the $P$-orbit $P \cdot x$ is $M$-\emph{admissible} if $ M =  w \theta(M) w^{-1}$ where $w  = \iota_M (P \cdot x )$.
\end{defn}
We make the following definition following the terminology in \cite{Mitra-Offen-U2n-Sp2n}. Let $\CO$ be a $P$-orbit in $X$.
\begin{defn}
	We say that $x \in \CO$ is a good representative if $ x \in Lw$ where $ w = \iota_M (\CO)$.
\end{defn}
By \cite[Lemma 3.2]{Offen-ParabolicInduction-JNT}, for every $P$-orbit, good representatives always exist. Assume that $x \in Lw$ is a good representative for the $P$-orbit $\CO$ and that $w = \iota_M (\CO)$. Let $Q = LV$ be the standard parabolic subgroup with Levi $L$. Since $w$ is also left $W_L$- and right $W_{\theta(L)}$-reduced, we have $\iota_L (Q\cdot x ) = w $. It is easily seen that $x$ is $L$-admissible and 
that $x$ is a good representatitive of the $Q$-orbit $Q\cdot x $.

In summary, to each $P$-orbit $\CO$ in $G \cdot e$, we will attach the following data:
\begin{itemize}
	\item the Weyl element $w = \iota_M (\CO) $ so that $w$ is left $W_M$-reduced and right $W_{\theta(M)}$-reduced and satisfies $w\theta(w) = e$ 
	\item the standard parabolic subgroup $Q = LV$ of $P$ with $L = M \cap w\theta(M)w^{-1}$
	\item  a good representative $x$ in $\CO$
	\item  the modular character $\delta_x : = \delta_{Q_x}\delta_Q^{-1/2}$
\end{itemize}
The role of the modular character $\delta_x$ will be clear from the next proposition. Assume that we have chosen the data as above for each $P$-orbit $\CO$.

Let $\chi$ be a character of $A_{G,H}$ and $\tilde{\chi}$ be a character of $Z$ such that $\tilde{\chi}$ and $\chi$ coincides on $A_{G,H}$. Assume that $\sigma$ is a smooth representation of $M$ on which $Z$ acts via $\tilde{\chi}$. 

\begin{prop}\label{prop::Geometric-lemma}
	We have
	\begin{align}\label{formula::geometric-lemma}
		\Ext{q}_{\CM^{\,\chi}(H,\,A_{G,H})} (\pi_{\CO},\chi) = \Ext{q}_{ \CM^{\,\chi} (L_x,\,A_{G,H})} (r_{L,M}(\sigma),\delta_x \chi^{\eta^{-1}}) \text{ for all } q\geqslant 0.
	\end{align}
	Here $r_{L,M}$ is the normalized Jacquet module as defined in \cite[Section 2.3]{B-Z-I}.
\end{prop}
\begin{proof}
	The case $q = 0$ was proved in \cite[Proposition 4.1]{Offen-ParabolicInduction-JNT}. The simple observation here is that the proof there can be adpated to the case of higher extension groups with no extra efforts. We sketch the proof for the convenience of the reader. By \eqref{formula::composition factor} and the Frobenius reciprocity in Lemma \ref{lem::Frobenius}, we have
	\begin{align*}
		\Ext{q}_{\CM^{\,\chi}(H,\,A_{G,H})} (\pi_{\CO},\chi) = \Ext{q}_{\CM^{\,\chi} (H^P_{\eta},\,A_{G,H})} ((\sigma \delta_P^{1/2})^{\eta}, \delta_{H^P_{\eta}} \shskip \chi) = \Ext{q}_{\CM^{\,\chi} (P_x,\,A_{G,H})} (\sigma \delta_P^{1/2}, \delta_{P_x}  \chi^{\eta^{-1}}).
	\end{align*}
	By \cite[Lemma 3.2]{Offen-ParabolicInduction-JNT}, $P_x = L_x \ltimes \tilde{U}$, where $L$ is chosen as above and $\tilde{U}$ is a unipotent subgroup in $P_x$. By Lemma \ref{lem::Frobenius} again, we have
	\begin{align*}
		\Ext{q}_{\CM^{\,\chi} (P_x,\,A_{G,H})} (\sigma \delta_P^{1/2}, \delta_{P_x}  \chi^{\eta^{-1}}) = \Ext{q}_{\CM^{\,\chi} (L_x,\,A_{G,H})} ( \delta_P^{1/2} \cdot (\sigma)_{\tilde{U}} , \delta_{P_x}  \chi^{\eta^{-1}}),
	\end{align*} 
	where $(\sigma)_{\tilde{U}}$ is the $\tilde{U}$-coinvariants in $\sigma$ with ordinary action by $L_x$. By \cite[Lemma 3.3]{Offen-ParabolicInduction-JNT} , the projection of $\tilde{U}$ to $M$ under the natural projection map $P \ra M$ is $M \cap V$, hence the underlying space of $(\sigma)_{\tilde{U}}$ is exactly $r_{L,M}(\sigma)$. By the proof of \cite[Proposition 4.1]{Offen-ParabolicInduction-JNT}, we have $P_x = Q_x$. Taking into account the normalization factors in the definition of Jacquet modules, we can finish the proof of \eqref{formula::geometric-lemma}.
\end{proof}

For any finite length smooth representation $\pi$ of $G$, by  \cite{Aizenbud-Sayag-HomologicalMulti}, we know that $\Ext{i}_H (\pi,\BC)$ is of finite dimensional for all $ i \geqslant 0$.

\begin{defn}
	Let $\sigma$ be a finite length smooth representation of $M$ on which $A_{G,H}$ acts trivially. Let $\pi = I_P^G (\sigma)$.
	
	$(1)$ We say that the $P$-orbit $\CO$ is relevant to $\Ext{\sharp}_{H/ A_{G,H}} (\pi,\BC)$ resp. $\EP_{H /A_{G,H}} (\pi,\BC)$ if 
	\[ \Ext{\sharp}_{ H / A_{G,H}} (\pi_{\CO},\BC ) \neq \{0\}, \text{ resp. } \EP_{H / A_{G,H}} (\pi_{\CO},\BC) \neq 0.\]
	
	$(2)$ We say that $\Ext{\sharp}_{H/ A_{G,H}} (\pi,\BC) $ resp. $ \EP_{H /A_{G,H}} (\pi,\BC) $ is supported on a set $S$ of orbits if any orbit $\CO'$ not contained in $S$ is not relevant to $\Ext{\sharp}_{H/ A_{G,H}} (\pi,\BC) $ resp. $ \EP_{H /A_{G,H}} (\pi,\BC) $.
\end{defn}

\section{Representation theory of $\GL_n(F)$}\label{sec::GLn}

In this section we set up notations and recall some well known facts for representations of $\GL_n(F)$.

\subsection{Notations}

For any $ n \in \BN$, let $G_n = \GL_n(F)$. By convention we define $G_0$ as the trivial group. Let $|\cdot|_F$ be the normalized absolute value on $F$. Let $\nu$ be the character $\nu(g)  =  |\det g \shskip |_F$ on any $G_n$. (The $n$ will be implicit and hopefully clear from the context.) Without further mentioning, by a representation, we always mean a smooth complex valued representation. For any representation $\pi$ of $G_n$ and $a \in \BR$, let $\nu^a \pi$ be the representation obtained from $\pi$ by twisting it by the character $\nu^a$. Let $\pi^{\vee}$ be the contragredient of $\pi$.

Fix $n$ and let $G = G_n$.  The standard parabolic subgroups of $G$ are in bijection with compositions $ (n_1,  \cdots ,n_t)$ of $n$. For a composition $\alpha = (n_1,\cdots,n_t)$ of $n$ let $P_{\alpha} = M_{\alpha} \ltimes U_{\alpha}$ be the standard parabolic subgroup of $G$ consisting of block uppertriangular matrices with standard Levi subgroup
\begin{align*}
	M_{\alpha}  =  \{ \text{diag}(g_1,\cdots,g_t) \colon g_i \in G_{n_i}, i = 1,\cdots,t \} \cong  G_{n_1} \times \cdots \times G_{n_t}
\end{align*}
and unipotent radical $U_{\alpha}$. We will say that $P_{\alpha}$ and $M_{\alpha}$ are of type $\alpha$. If $\rho_1, \cdots, \rho_t$ are representations of $G_{n_1}, \cdots, G_{n_t}$ respectively , we denote by 
\begin{align*}
	\rho_1 \times  \cdots \times \rho_t  
\end{align*}
the induced representation $I_{P_{\alpha}}^{G} (\sigma)$ where $\sigma$ is the representation $\rho_1  \otimes\cdots \otimes \rho_t$ of $M_{\alpha}$.

Let $T$ be the subgroup of diagonal matrices and $B$ the subgroup of uppertriangular matrices in $G$. The Weyl group $W = N_G(T) /T$ of $G$  can be naturally identitfied with the permutation group $\mathfrak{S}_n$ of $\{1,2,\cdots, n\}$. We also identify $W$ with the subgroup of permutation matrices in $G$. Let $w_n = (\delta_{i,n+1-i}) \in W$ be the longest Weyl element. For the standard Levi subgroup $M$ of type $(n_1,n_2,\cdots,n_t)$, we set
\begin{align*}
	w_M = \begin{pmatrix}
		&  & I_{n_1} \\
		& \iddots & \\
		I_{n_t} & & 
	\end{pmatrix}.
\end{align*}

\subsection{Segments}

By a segment of cuspidal representations we mean a set of the form
\begin{align}\label{eq::GL-segment}
	\Delta  =  \{\nu^a \rho, \nu^{a+1}\rho, \cdots, \nu^b\rho  \},
\end{align}
where $\rho $ is a supercuspidal representation of $G_d$, for some $d$, and $a, b \in \BR$, $b- a \in \BZ_{\geqslant 0}$. The representation 
\begin{align}\label{eq::GL-induced}
	\nu^a\rho \times \nu^{a+1}\rho \times \cdots \times \nu^b\rho
\end{align}
has a unique irreducible quotient, which is an essentially square-integrable representaton and is denoted by $\textup{L}(\Delta)$.
The map $\Delta \mapsto \textup{L}(\Delta)$ gives a bijection between the set of segments of cuspidal representations and the set of equivalent classes of essentially square-integrable representations. The induced representation \eqref{eq::GL-induced} also has a unique irreducible subrepresentation, which will be denoted by $\textup{Z} (\Delta)$. When $a = -(k-1)/2$ and $b = (k-1)/2$ for some positive integer $k$, then $\textup{L}(\Delta)$ is the so-called generalized Steinberg representation and will be denoted by $\St_k(\rho)$.

For the segment $\Delta$ in \eqref{eq::GL-segment}, we define its length $l (\Delta)$ to be $b-a+1$. The dual segment $\Delta^{\vee}$ is defined by
\begin{align*}
	\Delta^{\vee}  =  \{ \nu^{-b} \rho^{\vee}, \nu^{-b+1} \rho^{\vee}, \cdots, \nu^{-a} \rho^{\vee}  \}.
\end{align*}
So we have $\textup{L}(\Delta)^{\vee} = \textup{L}(\Delta^{\vee})$ and $\textup{Z} (\Delta)^{\vee}  = \textup{Z} (\Delta^{\vee})$.

Let $\Delta_1$ and $\Delta_2$ be two segments. We say that $\Delta_1$ and $\Delta_2$ are linked if $\Delta_1 \cup \Delta_2$ forms a segment but neither $\Delta_1 \subset \Delta_2$ nor $\Delta_2 \subset \Delta_1$. For the segment $\Delta$ in \eqref{eq::GL-segment}, we denote the representation $\nu^a \rho$ by $\mathbf{b}(\Delta)$. If $\Delta_1$ and $\Delta_2$ are linked and $\mathbf{b}(\Delta_1) = \nu^j \mathbf{b}(\Delta_2)$ with $j < 0$, then we say that $\Delta$ precedes $\Delta'$ and write $\Delta \prec \Delta'$.

The description of Jacquet modules of $\textup{L}(\Delta)$ and $\textup{Z}(\Delta)$ will be useful for us. We recall them here (See \cite[Proposition 3.4, Proposition 9.5]{Zelevinsky-II}). For a segment $\Delta$, let
\begin{align*}
	\Delta = \Delta_1 \sqcup \Delta_2 \sqcup \cdots \sqcup \Delta_s
\end{align*}
be a partition of $\Delta$ into a disjoint union of segments $\Delta_1,\cdots,\Delta_s$ such that $\Delta_i \prec \Delta_j$ whenever $ i < j$. Then the Jacquet modules of $\textup{Z} (\Delta)$ are of the form
\begin{align*}
	\textup{Z}(\Delta_1) \otimes \textup{Z} (\Delta_2) \otimes \cdots \otimes \textup{Z} (\Delta_s),
\end{align*}
while the Jacquet modules of $\textup{L}(\Delta)$ are of the form (in reverse order)
\begin{align*}
	\textup{L}(\Delta_s) \otimes \textup{L} (\Delta_{s-1}) \otimes \cdots \otimes \textup{L} (\Delta_1).
\end{align*}

A multisegment is a multiset (that is, set with multiplicities) of segments. An order $\frm = \{\Delta_1,\cdots, \Delta_t \}$ on a multisegments $\frm$ is said to be standard if $\Delta_i \nprec\Delta_j$ for all $i < j$. Every multisegment $\frm $ admits at least one standard order. Let $\frm = \{\Delta_1, \cdots, \Delta_t  \}  $ be ordered in standard form. The representation 
\begin{align*}
	 \textup{L}( \Delta_1) \times \cdots \times \textup{L}(\Delta_t)
\end{align*}
is independent of the choice of order of standard form and has a unique irreducible quotient that we denote by $\textup{L}(\frm)$. The representation
\begin{align*}
	\textup{Z}( \Delta_1) \times \cdots \times \textup{Z}(\Delta_t)
\end{align*}
is also independent of the choice of order of standard form and a unique irreducible subrepresentation that we denote by $\textup{Z} (\frm)$.  

The following result can be extracted easily from results in \cite{Zelevinsky-II} and is useful for our later computations.

\begin{prop}\label{prop::Zel-ShortExactSeq}
	Let $\Delta_1$ and $\Delta_2$ be two segments such that $\Delta_1$ precedes $\Delta_2$. Let $\Delta = \Delta_1 \cup \Delta_2$ and $\Delta' = \Delta_1 \cap \Delta_2$. Then we have two sets of short exact sequences
	\begin{align*}
		0 \ra \textup{Z}  (\Delta) \times \textup{Z}(\Delta')  \ra \textup{Z} (\Delta_1) \times \textup{Z}(\Delta_2)  \ra \textup{Z} (\Delta_1,\Delta_2) \ra 0,\\
		0 \ra \textup{Z} (\Delta_1,\Delta_2)   \ra \textup{Z} (\Delta_2) \times \textup{Z}(\Delta_1)  \ra  \textup{Z} (\Delta)\times \textup{Z}(\Delta')  \ra 0,
	\end{align*}
	and
	\begin{align*}
		0 \ra \textup{L} (\Delta_1,\Delta_2)  \ra \textup{L} (\Delta_1) \times \textup{L}(\Delta_2)  \ra \textup{L} (\Delta) \times \textup{L}(\Delta') \ra 0,\\
		0 \ra \textup{L} (\Delta)  \times \textup{L}(\Delta') \ra \textup{L} (\Delta_2) \times \textup{L}(\Delta_1)  \ra  \textup{L} (\Delta_1,\Delta_2) \ra 0,
	\end{align*}
\end{prop}




\section{The symplectic pair}\label{sec::symplectic}

For $n \in \BN$, let $G_{n} = \GL_{n}(F)$. Let 
\begin{align*}
	H_{2n} = \Sp_{2n} (F)  = \{ g \in G_{2n} \colon {}^t g  J_{2n} g = J_{2n}\},
\end{align*}
where 
\begin{align*}
	J_{2n}  =  \begin{pmatrix}
		  & w_n \\
		-w_n & 
	\end{pmatrix}.
\end{align*}
Let $G = G_{2n}$ and $H = H_{2n}$. Then $A_{G,H} = \{e\}$. Let $\theta$ be the involution on $G$ defined by
\begin{align*}
	\theta(g) = J \,\transinv{g} J^{-1}
\end{align*}
so that $H = G^{\theta}$.   

\subsection{The orbit analysis}

Let 
\begin{align*}
	X = \{ x \in G \colon \theta(x)x = e \}
\end{align*}
be the symmetric space associated to $(G,\theta)$ with the $G$-action given by $g \cdot x = g x \theta(g)^{-1}$. Note that $X J$ is the space of skew-symmetric spaces in $G$ and $(g \cdot x) J = g (xJ) \transpose{g}$. It follows that $X = G \cdot e$ is a single $G$-orbit.

The analysis of parabolic orbits in $X$ has been given in \cite[Section 3]{Offen-symplectic-disc-spectrum-IsraelJournal} and \cite[Section 3.1]{Offen-Residual-Duke}, to which we refer the reader for proofs of the results in this section. We follow mainly the exposition in \cite{Mitra-Offen-Sayag-Klyachko}.

Let $W$ be the Weyl group of $G$ and 
\begin{align*}
	[w_{2n}] = \{  w w_{2n}w^{-1} \colon w \in W\}
\end{align*}
be the conjugacy class of $w_{2n}$ in $W$. It is the set of involutions without fixed points in $W$. Then the set of twisted involutions $\mathfrak{I}_0(\theta) = [w_{2n}] w_{2n}$.

Let $\alpha = (n_1,\cdots,n_k)$ be a composition of $2n$ and let $P = P_{\alpha} = MU$. Note that $\theta(P) = P_{(n_k,\cdots,n_1)}$ and $\theta(M) = w_{2n} M w_{2n}^{-1}  =  M_{(n_k,\cdots,n_1)}$. For the symplectic pair, the fibration map mentioned in Section \ref{sec::gemeotric lemma} is a bijection.
\begin{lem}
	There is a bijection
	\begin{align*}
		\iota_M \colon P \backslash X  \ra {}_MW_{\theta(M)}  \cap [w_{2n}] w_{2n}
	\end{align*}
	such that
	\begin{align*}
		P x \theta(P)  =  P \,\iota_M ( P \cdot x) \, \theta(P).
	\end{align*}
\end{lem}

For the later computations, it suffices for us to describe explicitly the relevant data for $M$-admissible orbits in $X$. Let
\begin{align*}
	S_2[\alpha] = \{ \tau \in \mathfrak{S}_k \colon \tau^2 = e,\  n_{\tau(i)} = n_i,\  i = 1,\cdots, k \text{ and } n_i \text{ is even if } \tau(i) = i\}.
\end{align*}
For any $d \in \BN$, let ${}^{\star}$ denote the involution on $G_d$ defined by $g^{\star} = w_d \transinv{g} w_d^{-1}$.
\begin{lem}\label{lem::admissible orbit-Sp}
	There is a bijection between the $M$-admissible $P$-orbits in $X$ and $S_2[\alpha]$ that satisfies the following properties. For an $M$-admissible orbit $\CO$, let $ w = \iota_M(\CO)$ and $\tau \in S_2[\alpha] $ correspondes to $\CO$. Then there exists a good representative $x \in \CO \cap Mw $ such that:
	\begin{itemize}
		\item[(1)] $M_x = \{ \textup{diag}(g_1,\cdots,g_k) \colon g_{\tau(i)} = g_i^{\star} \text{ if }\tau(i) \neq i \text{ and }g_i \in H_{n_i} \text{ if }\tau(i) = i\}.$ 
		\item [(2)] $(\delta_{P_x} \delta_P^{-1/2}) (\textup{diag} (g_1,\cdots,g_k))  = \prod_{ i < \tau(i)}  \nu (g_i)$.
	\end{itemize} 
\end{lem}

Every $\tau \in \mathfrak{S}_k$ defines a uniqaue $w_{\tau} \in W$ such that for every $g= \textup{diag}(g_1,\cdots,g_k) \in M$ we have
\begin{align*}
	w_{\tau} g w_{\tau}^{-1} = \textup{diag}(g_{\tau^{-1}(1)}, \cdots, g_{\tau^{-1}(k)}).
\end{align*}
We can further make the relation between $w = \iota_M(\CO)$ and $\tau \in S_2[\alpha]$ corresponding to $\CO$ explicit. In fact, we have
\begin{align}\label{formula::admissible-involution-permutation-Sp}
	\textup{diag} (w_{n_1},\cdots, w_{n_k})  w_{\tau} w_{2n} = w.
\end{align}

\subsection{Consequences of the geometric lemma}

We formulate a useful result on the extension groups of parabolically induced representations.
\begin{lem}\label{lem::cor-geo-lemm-Sp}
	Let $P = MU$ be a standard parabolic subgroup of $G$ and $\sigma$ a finite length smooth representation of $M$. Let $\pi = I_P^G (\sigma)$.  For a $P$-orbit $\CO$ let $w,L,Q$ and $x$ be the data associated to $\CO$ as in Section \ref{sec::gemeotric lemma}. We further take $x$ as a good representative of $Q\cdot x$ as in Lemma \ref{lem::admissible orbit-Sp}. Assume that $L$ is of type $(n_1,\cdots,n_k)$. Then:
	
	 \textup{(1)}$\quad$ If the orbit $\CO$ is relevant to $\Ext{\sharp}_H (\pi,\BC)$,then either all $n_i$ are even, $i = 1,\cdots,k$, $M = L$ and $\CO$ is the unique $M$-admissible orbit such that $\iota_M (\CO) = w_M$, or there exists a composition factor $\rho_1 \otimes \cdots \otimes \rho_k$ of $r_{L,M}(\sigma)$ such that
	 \begin{align}\label{eq::cor-geo-lem-Sp-1}
	 	w_{\rho_i} = w_{\rho_{\tau(i)}} |\cdot|^{n_i}_F
	 \end{align}
	 for any $i$ with $i < \tau(i)$, where $w_{\rho}$ stands for the central character of $\rho$ and $\tau \in \mathfrak{S}_k$ is the permutation corresponding to the $L$-admissible orbit $Q \cdot x$ as in Lemma \ref{lem::admissible orbit-Sp}.
	 
	 \textup{(2)}$\quad $ If the orbit $\CO$ is relevant to $\EP_H(\pi,\BC)$, then all $n_i$ are even, $i = 1,\cdots,k$, $M = L$, and  $\CO$ is the unique $M$-admissible orbit such that $\iota_M (\CO) = w_M$.
\end{lem}
\begin{proof}
	For \textup{(1)}, by definition and Proposition \ref{prop::Geometric-lemma}, if $\CO$ is relevant, then we have 
	\begin{align*}
			\Ext{\sharp}_{L_x} (r_{L,M}(\sigma), \delta_x) \neq 0.
	\end{align*}
	Thus there is a composition factor $\rho_1 \otimes \cdots \otimes  \rho_k$ such that 
	$$\Ext{\sharp}_{L_x} (\rho_1 \otimes \cdots \otimes \rho_k, \delta_x ) \neq 0.$$
	 If there is some $i$ such that $\tau(i) \neq i$, then take $z = \textup{diag}(g_1,\cdots,g_k)$ such that for $i$ with $i < \tau(i)$, $g_i = \lambda_i I_{n_i}$ and $g_{\tau(i)} = \lambda_i^{-1} I_{n_i}$ with $\lambda_i \in F$, and that for $i$ with $\tau(i) = i$, $g_i = I_{n_i}$. We have that $z$ is central in $L_x$. Then \eqref{eq::cor-geo-lem-Sp-1} follows from Lemma \ref{lem::vanishing} and the description of $\delta_x$ in Lemma \ref{lem::admissible orbit-Sp}. If for all $i = 1,\cdots,k$, $\tau(i) = i$, then all $n_i$ are even. By \eqref{formula::admissible-involution-permutation-Sp}, we have $ w = \textup{diag} (w_{n_1},\cdots, w_{n_k}) w_{2n} $. Note that $w = \iota_M(\CO)$ is both left $W_M$- and right $W_{\theta(M)}$-reduced, we must have $L = M$ so that $\CO$ is $M$-admissible and $\iota_M (\CO) = w_M$.
	
	For (2), by definition and Proposition \ref{prop::Geometric-lemma}, if $\CO$ is relevant, then 
	\begin{align*}
		\EP_{L_x} (r_{L,M} (\sigma),\delta_x) \neq 0.
	\end{align*}
	Note that by our assumption on $\sigma$, $r_{L,M}(\sigma) $ is a finite length representation of $L$. By Proposition \ref{prop::EP-trivial} we then have $\tau(i) = i$ for all $i = 1,\cdots, k$. So part $(2)$ follows by the same argument as above.
\end{proof}

 \subsection{Extensions of essentially square-integrable representations}

In the symplectic model, we have the multiplicity one theorem and disjointness of Whittaker models and symplectic models due to Heumos and Rallis \cite{Heumos-Rallis}.
\begin{lem}\label{lem::sp-multi-one}
	Let $\pi$ be an irreducible representation of $G$. Then $\Hom_H (\pi,\BC) \leqslant 1$. When $\pi$ is generic, then $\Hom_H (\pi,\BC) = 0$.
\end{lem}

For a smooth finite length representation $\pi$ of $G_{2d}$, if $\Hom_{H_{2d}} (\pi, \BC) \neq 0$, then we say simply that $\pi$ is $\Sp$-distinguished.

Let $\rho$ be an irreducible cuspidal representation of $G_d$ and 
$$\Delta = \{ \rho, \nu \rho, \cdots, \nu^{k-1} \rho\}$$
 be a segment ($2n = k d$). We begin by computing the extension groups of $Z(\Delta)$.

\begin{prop}\label{prop::speh-ext}
	We have
	\begin{align*}
 		\Ext{i}_H (Z(\Delta), \BC) = \begin{cases*}
				                       \BC,\quad i =0\text{ and }2 | k, \\
				                       0,\quad \text{otherwise.}
			  						 \end{cases*}
	\end{align*}
\end{prop}
 \begin{proof}
 	 If $k = 1$, then $d = 2n$ is even and $Z(\Delta)$ is just the cuspidal $\rho$. By Lemma \ref{lem::sp-multi-one} we know that $\Ext{0}_H (\rho, \BC)$ = 0. Since $H $ is contained in the derived subgroup of $G_{2n}$, the restriction of $\rho$ to $H$ is projective. So we have $\Ext{i}_H (\rho, \BC) = 0$ for all $i  > 0$. We assume that $k \geqslant 2$. Write $\Delta_1 =\{ \rho\}$ and $\Delta_2 = \{ \nu \rho ,\cdots, \nu^{k-1} \rho\}$ so that $\Delta$ is the disjoint union of segments $\Delta_1$ and $\Delta_2$. By Proposition \ref{prop::Zel-ShortExactSeq}, we have two short exact sequences
 	\begin{align}\label{formula::exact-speh-I}
 			0 \ra Z(\Delta)  \ra \rho \times Z(\Delta_2)  \ra Z(\Delta_1,\Delta_2)  \ra 0
 	\end{align}	
 	and
 	\begin{align}\label{formula::exact-speh-II}
 			0 \ra Z(\Delta_1, \Delta_2) \ra Z(\Delta_2) \times \rho  \ra Z(\Delta) \ra 0 .
 	\end{align}
 	
 	We first show that 
 	\begin{align}\label{formula::vanishing-speh-I}
 		\Ext{i}_H ( \rho \times Z(\Delta_2), \BC) = 0 \text{ for all $i$.}
 	\end{align}
 	Let $P = MU$ be the parabolic subgroup of type $(d,(k-1)d)$. We show that no $P$-orbit $\CO$ is relevant to $\Ext{i}_H(\rho \times Z(\Delta_2),\BC)$. By the description of Jacquet modules of $Z(\Delta_2)$, we see easily that \eqref{eq::cor-geo-lem-Sp-1} does not hold for any composition factor of $r_{L,M} (\rho \otimes Z(\Delta_2))$. Thus, if $\CO$ is relevant, then by Lemma \ref{lem::cor-geo-lemm-Sp}, $d$ is even and $\CO$ is the $M$-admissible orbit corresponding to 
 	\begin{align*}
 		\begin{pmatrix}
 			& I_d \\
 		I_{(k-1)d} & 	
 		\end{pmatrix}.
 	\end{align*}
 	Note that in this case $L = M$ and the modular character $\delta_x$ is trivial by Lemma \ref{lem::admissible orbit-Sp}. By Proposition \ref{prop::Geometric-lemma} and the description of $L_x$ in Lemma \ref{lem::admissible orbit-Sp}, we have
 	\begin{align*}
 		\Ext{i}_H ((\rho \times Z(\Delta_2))_{\CO} , \BC) = \Ext{i}_{H_d \times H_{(k-1)d}} ( \rho \otimes Z(\Delta_2), \BC).
 	\end{align*}
 	By the Kunneth theorem Theorem \ref{thm::Kunneth} and the cuspidality of $\rho$, we have 
 	\[ \Ext{i}_{H_d \times H_{(k-1)d}} ( \rho \otimes Z(\Delta_2), \BC) = 0.  \]
 	It follows from  \eqref{formula::vanishing-speh-I} and the long exact sequence associated to \eqref{formula::exact-speh-I} that 
 	\begin{align}\label{formula::speh-long-exact-I}
 		\Ext{i}_H (Z(\Delta), \BC) = \Ext{i+1}_H (Z(\Delta_1,\Delta_2), \BC), \quad \text{for all } i \geqslant 0.
 	\end{align}
 	
 	We then compute $\Ext{i}_H (Z(\Delta_2) \times \rho, \BC)$ according to the parity of $k$. When $k$ is odd, we will show that 
 	\begin{align}\label{formula::vanishing-speh-II}
 		\Ext{i}_H (Z(\Delta_2) \times \rho, \BC) = 0, \quad \text{for all } i \geqslant 0.
 	\end{align}
 	Assuming \eqref{formula::vanishing-speh-II}, by \eqref{formula::speh-long-exact-I} and the long exact sequence associated to \eqref{formula::exact-speh-II}, we have
 	\begin{align*}
 		\Ext{i}_H (Z(\Delta),\BC) = \Ext{i+2}_H (Z(\Delta), \BC), \quad  i \geqslant 0.
 	\end{align*}
 	This proves the proposition when $k$ is odd since $\Ext{i}_H (Z(\Delta), \BC)$ vanishes when $i$ is large enough. We now prove \eqref{formula::vanishing-speh-II} by induction on $k$. The case $k = 1$ has been proved. Let $P = MU$ be the standard parabolic subgroup of type $((k-1)d,d)$. By part (1) of Lemma \ref{lem::cor-geo-lemm-Sp}, $\Ext{i}_H (Z(\Delta_2) \times \rho, \BC)$ is supported only on those $P$-orbits for which the associated standard Levi subgroup $L$ is of type $(d,r_1 d,\cdots,r_td,d)$ with $\sum r_j = k-2$. Since $k$ is odd, there exists a $j$ such that $r_j$ is odd. Then \eqref{formula::vanishing-speh-II} follows from Proposition \ref{prop::Geometric-lemma}, Theorem \ref{thm::Kunneth} and the induction hypothesis.
 	
 	 When $k$ is even, by \cite[Corollary 10.5]{Mitra-Offen-Sayag-Klyachko}, we have
 	 \begin{align*}
 	 	\Hom_H (Z(\Delta), \BC) = \BC.
 	 \end{align*}
 	 By an argument using the long exact sequence as above, to finish the proposition, it suffices for us to show that
 	\begin{align}\label{formula::speh-product}
 		\Ext{i}_H ( Z(\Delta_2) \times \rho ,\BC) = \begin{cases*}
 			\BC, \quad i = 0,1\\
 			0, \quad \text{otherwise.}
 		\end{cases*}
 	\end{align}
	Let $P = MU$ be the standard parabolic subgroup of type $((k-1)d,d)$. By part (1) of Lemma \ref{lem::cor-geo-lemm-Sp}, $\Ext{i}_H(Z(\Delta_2)\times \rho, \BC)$ is supported only on those $P$-orbits $\CO$ such that $L$ is of type $(d,r_1 d,\cdots,r_t d,d)$. Let $x,w,Q$ and $\tau \in \mathfrak{S}_{t+2}$ be as in Lemma \ref{lem::cor-geo-lemm-Sp}. We further have that $\tau$ interchanges $1$ and $t+2$ and fix any other $i$, $1<i<t+2$. By \eqref{formula::admissible-involution-permutation-Sp}, we have
 	\begin{align*}
 	  w =	\begin{pmatrix}
 			  & & & & w_{d} \\
 			  & w_{r_1d} & & & \\
 			  & & \ddots & &  \\
 			  & & & w_{r_td}  &  \\
 			  w_d & & & & 
 		\end{pmatrix} w_{2n}.
 	\end{align*}
 	Note that $w$ is both left $W_M$- and right $W_{\theta(M)}$-reduced, we must have $w = I_{2n}$. Thus $\Ext{i}_H (Z(\Delta_2) \times \rho, \BC)$ is supported only on the obit $P\cdot e$ and $L$ is of type $(d, (k-2)d,d)$. Set $\Delta_3 = \{ \nu \rho, \cdots, \nu^{k-2} \rho\}$. Then by Proposition \ref{prop::Geometric-lemma} and Theorem \ref{thm::Kunneth}, we have
	\begin{equation}\label{formula::speh-product-Kunneth}
		\begin{aligned}
			\Ext{i}_H (Z(\Delta_2) \times \rho , \BC) &=  \Ext{i}_{G_d \times H_{(k-2)d}} ((\nu \rho \otimes\rho^{\vee}) \boxtimes Z(\Delta_3) , \nu \boxtimes \mathbf{1})  \\
			&=\mathop{\oplus}_{j=0}  \Ext{j}_{G_d} (\rho, \rho) \otimes \Ext{i-j}_{H_{(k-2)d}} (Z(\Delta_3),\BC).
		\end{aligned}
	\end{equation}
 	The first identiy uses the fact that the involution ${}^{\star}$ takes an irreducible representation to its smooth dual and the second identity uses Lemma \ref{lem::tensor-hom-ext}. We can now prove \eqref{formula::speh-product} by induction on $k$. If $k = 2$, by \eqref{formula::speh-product-Kunneth} we have 
 	$$\Ext{i}_H(\nu \rho \times \rho, \BC) = \Ext{i}_{G_d} (\rho, \rho).$$
   Since $\rho$ is a cuspidal representation, by \cite[Proposition 2.9, Lemma 4.1]{Prasad-Ext-BranchingLaw-ICM2018}, we have 
 	\begin{align}\label{formula::ext-GL-same}
 		\Ext{i}_{G_d} (\rho ,\rho) = \begin{cases*}
 			\BC, \quad i = 0,1\\
 			0 , \quad \text{otherwise.}
 		\end{cases*}
 	\end{align}
 	for all $d$. For $k \geqslant 4$, we have already seen that the induction hypothesis implies that $\Ext{i}_{H_{(k-2)d}} (Z(\Delta_3), \BC)  =0 $ for all $i \geqslant 1$. Since the length of $\Delta_3$ is even, by \cite[Corollary 10.5]{Mitra-Offen-Sayag-Klyachko}, $\Ext{0}_{H_{(k-2)d}} (Z(\Delta_3), \BC) = \BC$. Therefore, \eqref{formula::speh-product} follows easily from \eqref{formula::speh-product-Kunneth} and \eqref{formula::ext-GL-same}.
 \end{proof}

 Let $\rho$ and $\Delta$ be as above. Every essentially square-integrable representation of $G$ is of the form $\textup{L}(\Delta)$, that is, the unique irreducible quotient of
 \begin{align*}
 	\rho \times \nu \rho \times \cdots \times \nu^{k-1} \rho.
 \end{align*}
Let $P = MU$ be the parabolic subgroup of type $(d,d,\cdots,d)$ and let $T_M$ be the maximal split torus contained in the center of $M$. Let $\alpha_i$ be the simple root relative to $M$ defined by $\alpha_i (t) = t_i/ t_{i+1}$ where 
 \begin{align*}
 	t = (\underbrace{t_1, \cdots,t_1}_d, \cdots, \underbrace{t_k,\cdots,t_k}_d) \in T_M
 \end{align*}
 for $i = 1,2,\cdots, k-1$. Denote by $S_M$ the set of simple roots relative to $M$ above so that $S_M$ can be naturally identified with the set $\{1,2,\cdots,k-1\}$. To any subset $I \subset \{1,\cdots,k-1\}$, one can associate a standard Levi subgroup $M_I$ that is generated by $M$ and the simple roots corresponding to $I$. Let $P_I$ be the standard parabolic subgroup of $G$ with standard Levi $M_I$. One can also partition the segment $\Delta$ in a unique manner into a disjoint union $\sqcup_i \Delta_{I,i}$ of $k - |I|$ segments such that $\Delta_{I,i}$ preceeds $\Delta_{I,j}$ whenever $i < j$ and that
 \begin{align*}
 	\sigma_I = Z(\Delta_{I,1}) \otimes Z(\Delta_{I,2}) \otimes \cdots \otimes Z(\Delta_{I,k-|I|})
 \end{align*}
  is a representation of $M_I$. For example, if $I = \{ 1,2\}$, then  
 \begin{align*}
 	\sigma_I = Z([\rho, \nu \rho, \nu^2 \rho]) \otimes \nu^3 \rho \otimes \cdots \otimes \nu^{k-1} \rho.
 \end{align*}
 If $I = \emptyset$, then
 \begin{align*}
 	\sigma_{\emptyset} = \sigma = \rho \otimes \nu \rho \otimes \cdots \otimes \nu^{k-1} \rho.
 \end{align*}
 
 In order to compute the extension groups of $\textup{L}(\Delta)$, we need to introduce a resolution of $ \textup{L}(\Delta) $ by induced representations. First, by \cite[\S 1.6]{Zelevinsky-II} and \cite[Theorem 2.8]{Zelevinsky-II}, we know that for all subset $I \subset S_M$,
 \begin{align*}
 	I_{P_I}^G \sigma_I =  Z(\Delta_{I,1}) \times Z(\Delta_{I,2}) \times \cdots \times Z(\Delta_{I,k-|I|})
 \end{align*}
 is a subrepresentation of $I_P^G \sigma$. For $I \subset J$, we also have $I_{P_J}^G \sigma_J \subset I_{P_I}^G \sigma_I$. We fix a compatible inclusions for these representations. For arbitary subset $I, J \subset \{ 1,2,\cdots,k-1\}$ with $|J| - |I| = 1$, we let
 \begin{align*}
 	d_{J,I}  =  \begin{cases*}
 		(-1)^{n_{J,I}},  \quad  I \subset J \\
 		0,   \qquad \quad \ \   I  \nsubseteq J
 	\end{cases*}
 \end{align*}
 where $n_{J,I}$ is the number of pairs $(i,j)$ with $i >j$, $i \in I$ and $j \in J\backslash I$.
 
 \begin{lem}\label{lem::resolution-st}
 	The complex
 	\begin{align*}
 		0 \ra I_G^G \sigma_{S_M} \ra \! \! \!\bigoplus_{\stackrel{I \subset S_M}{|S_M\backslash I| = 1}} I_{P_I}^G \sigma_I \ra \!\! \!\bigoplus_{\stackrel{I \subset S_M}{|S_M\backslash I| = 2}} I_{P_I}^G \sigma_I  \ra \cdots \ra \bigoplus_{\stackrel{I \subset S_M}{| I| = 1}} I_{P_I}^G \sigma_I \ra I_P^G \sigma \ra \textup{L}(\Delta) \ra 0
 	\end{align*}
 	with differentials induced by $d_{J,I}$ above is exact.
 \end{lem}
 \begin{proof}
 	The proof is similar to that of \cite[Proposition 11]{Orlik-Extensions-Steinberg-JoA}. By \cite[Proposition 2.10]{Zelevinsky-II}, the complex is exact at the last two nonzero terms. Next we apply \cite[\S 2, Proposition 6]{Schneider-Stuhler-Cohomology-Invention-91} to the subrepresentations $I_{P_I}^G \sigma_I$ with $I \subset S_M$ and $|I| = 1$.	The conditions in \emph{loc.cit} is fulfilled once we can show that 
 	\begin{align*}
 		I_{P_I}^G \sigma_I \cap I_{P_J}^G \sigma_J = I_{P_{I \cup J}}^G \sigma_{I \cup J}
 	\end{align*}
 	and 
 	\begin{align*}
 		I_{P_I}^G \sigma_I  \cap (I_{P_J}^G \sigma_J  + I_{P_K}^G \sigma_K ) = ( I_{P_I}^G \sigma_I  + I_{P_J}^G \sigma_J ) \cap (I_{P_I}^G \sigma_I  + I_{P_K}^G \sigma_K ).
 	\end{align*}
 	The second identity follows from \cite[Proposition 2.4]{Zelevinsky-II}. The first identity follows also from \cite[Proposition 2.4]{Zelevinsky-II} and the description of composition factors of $I_{P_I}^G \sigma_I $ in terms of orientations on the graph in \cite[\S 2.2]{Zelevinsky-II}.									
 \end{proof}

 We next compute the extensions of these induced representations in the acyclic complex. For $I \subset \{1,\cdots,k-1\}$, we will say that $I$ is of even type if $k$ is even and the complement of $I$ consists of even numbers. Note that this is equivalent to the condition that all $l(\Delta_{I,i})$ are even for $i = 1, \cdots, k - |I|$.
 \begin{prop}\label{prop::product of speh}
 	We have 
 	\begin{align*}
 		\Ext{i}_H (I_{P_I}^G \sigma_I ,\BC) =0
 	\end{align*}
 	unless $i = 0$ and $I$ is of even type. While in this case, we have
 	\begin{align*}
 		\Ext{0}_H (I_{P_I}^G \sigma_I ,\BC) = \BC.
 	\end{align*}
 \end{prop}
 \begin{proof}
 	By the definition of $\sigma_I$, we see easily that \eqref{eq::cor-geo-lem-Sp-1} cannot happen. If one of the $l(\Delta_{I,i})$ is odd, then by part (1) of Lemma \ref{lem::cor-geo-lemm-Sp}, the Kunneth theorem \ref{thm::Kunneth} and Proposition \ref{prop::speh-ext}, 
 	\begin{align*}
 		\Ext{i}_H (I_{P_I}^G \sigma_I ,\BC) = 0  \text{ for all }i \geqslant 0.
 	\end{align*}
 	  Now we assume that all of them are even. By part $(1)$ of Lemma \ref{lem::cor-geo-lemm-Sp}, We have that $\Ext{i}_H (I_{P_I}^G \sigma_I ,\BC)$ is supported on the unique $M_I$-admissible orbit which corresponds to 
 	  $$\textup{diag}(w_{n_1d},\cdots, w_{n_td}) w_{2n},$$
 	  where $t = k - |I|$ and $n_i = l (\Delta_{I,i})$ are all even. The proposition then follows immediately from Proposition \ref{prop::Geometric-lemma}, the Kunneth theorem \ref{thm::Kunneth} and Proposition \ref{prop::speh-ext}.
 \end{proof}
 
 For $ I \subset J$ with $|J \backslash I| = 1$, we denote by $K_{J,I}$ the representation sitting in the exact sequence
 \begin{align}\label{formula::exact-ind-quotient}
 	0 \ra I_{P_J}^G \sigma_J \ra I_{P_I}^G \sigma_I \ra K_{J,I} \ra 0. 
 \end{align} 
 
\begin{lem}\label{lem::ext-quotient}
	If $I$ is of even type, we have 
	\begin{align*}
		\Ext{i}_H (K_{J,I}, \BC) = 0 \text{ for all } i \geqslant 0. 
	\end{align*}	
	If $I$ is not so, we have
	\begin{align*}
		\Ext{i}_H (K_{J,I},\BC) = 0 \text{ if } i \neq 1
	\end{align*}
	and 
	\begin{align*}
		\Ext{1}_H (K_{J,I},\BC) = \begin{cases*}
			\BC, \quad \text{if $J$ is of even type,} \\
			0, \quad \text{otherwise.}
		\end{cases*}
	\end{align*}
\end{lem}
 \begin{proof}
 	By Proposition \ref{prop::product of speh} and the long exact sequence associated to \eqref{formula::exact-ind-quotient}, we see readily that $\Ext{i}_H (K_{J,I}, \BC) = 0$ for $i \geqslant 2$. We also have an exact sequence 
 	\begin{align*}
 		0 \ra \Hom_H (K_{J,I}, \BC) \ra \Hom_H (I_{P_I}^G \sigma_I, \BC) \ra \Hom_H (I_{P_J}^G \sigma_J, \BC) \ra \Ext{1}_H (K_{J,I}, \BC) \ra 0.
 	\end{align*}
 	The case where $I$ is not of even type follows easily from this exact sequence. Now we assume that $I$ (hence $J$) is of even type. We know that $K_{J,I}$ is of the form
 	\begin{align*}
 		Z(\Delta_{I,1}) \times \cdots \times Z(\Delta_{I,i-1}) \times Z(\Delta_{I,i}, \Delta_{I,i+1}) \times Z(\Delta_{I,i+2}) \times \cdots \times Z(\Delta_{I,k-|I|}).
 	\end{align*}
 	Note that $\Delta_{I,i}$ preceeds $\Delta_{I,j}$ whenever $i < j$. If a $P_J$-orbit $\CO$ is relevant to $\Hom_H (K_{J,I} ,\BC) $, then $\CO$ is the $M_J$-admissilbe orbit such that $\iota_{M_J} (\CO)  = w_{M_J}$. This implies, among other things, that $Z (\Delta_{I,i},\Delta_{I,i+1})$ is $\Sp$-distinguished. But this is impossilbe by \cite[Corollary 10.5]{Mitra-Offen-Sayag-Klyachko}. Thus 
 	\begin{align*}
 		\Hom_H (K_{J,I},\BC) = 0
 	\end{align*}
 	It then follows from Proposition \ref{prop::product of speh} and the Euler-Poincar\'e principle that 
 	$$\Ext{1}_H (K_{J,I}, \BC) = 0.$$
 \end{proof}

 Our main result in this section is the following 
	\begin{thm}\label{thm::Ext-DS-GL-Sp}
		If $l(\Delta) \neq 2$, we have
		\begin{align*}
			\Ext{i}_H (  \textup{L}(\Delta), \BC) = 0\ \textup{ for all }i \geqslant 0.
		\end{align*}
		If $l(\Delta) = 2$, we have
		\begin{align*}
			\Ext{1}_H ( \textup{L}(\Delta), \BC) = \BC, \quad \Ext{i}_H ( \textup{L}(\Delta), \BC) = 0,\ i \neq 1.
		\end{align*}
	\end{thm}
	\begin{proof}
		We apply the resolution of $ \textup{L}(\Delta) $ in Lemma \ref{lem::resolution-st}. Taking an injective resolution of $\BC$ in the category of smooth representations of $H$ then gives in the usual way a double complex such that its associated spectral sequence converges to $\Ext{\sharp}_H ( \textup{L}(\Delta) , \BC)$. The $E_1$ page of this spectral sequence has the shape
		\begin{align*}
			0 &\ra \Ext{\sharp}_H (I_P^G \sigma, \BC) \ra \bigoplus_{\stackrel{I \subset S_M}{| I| = 1}} \Ext{\sharp}_H (I_{P_I}^G \sigma_I ,\BC) \ra\bigoplus_{\stackrel{I \subset S_M}{| I| = 2}} \Ext{\sharp}_H (I_{P_I}^G \sigma_I ,\BC) \ra  \cdots  \\
			&\ra \bigoplus_{\stackrel{I \subset S_M}{| S_M \backslash I| = 1}} \Ext{\sharp}_H (I_{P_I}^G \sigma_I,\BC) \ra  \Ext{\sharp}_H (I_G^G \sigma_{S_M}, \BC) \ra 0.
		\end{align*}
		By Proposition \ref{prop::product of speh} we see that the spectral sequence collapses at page $E_1$ and also that $\Ext{i}_H (\textup{L}(\Delta),\BC) = 0$ for all $i \geqslant 0$ when $k$ is odd. When $k$ is even, set $k_0 = k/2 -1$. The only row with nonzero terms in the $E_1$ page is 
		\begin{align}\label{formula::E-1-spectral-thm}
		  \cdots	\ra 0 \ra \BC^{k_0 \choose k_0} \ra \BC^{k_0\choose k_0 -1} \ra \cdots \ra \BC^{k_0 \choose 1} \ra \BC \ra 0\ra \cdots,
		\end{align}
		where the first nonzero term from the left occurs at the $k/2$-th place. It remains to determine the differential maps in \eqref{formula::E-1-spectral-thm}. By Lemma \ref{lem::ext-quotient} and Proposition \ref{prop::product of speh}, when $I \subset J$ with $I$ of even type and $|J \backslash I| = 1$, the restriction of a nonzero linear form on $I_{P_I}^G \sigma_I$ to $I_{P_J}^G \sigma_J$ is also nonzero. We then fix bases in these $\Ext{0}_H (I_{P_I}^G \sigma_I, \BC)$ compatible with the inclusions fixed before such that the differentials in \eqref{formula::E-1-spectral-thm} are the alternating sum of the identity map. Thus the cohomology groups vanish unless $k = 2$ ($k_0 = 0$) and the theorem follows easily.
	\end{proof}

  \subsection{The multiplicity formula}\label{sec::EP-symplectic}
  
  We  begin with some general notation. For a connected reductive group $\BG$ defined over $F$ and $G = \BG(F)$, we let $\CT_{ell}(G)$ be a set of representatives of maximal elliptic tori of $G$ and we let $\Gamma_{ell}(G)$ be the set of regular semisimple elliptic conjugacy classes of $G$. We define the Haar measure on $\Gamma_{ell}(G)$ to be
  $$\int_{\Gamma_{ell}(G)}f(X)dX=\sum_{T\in \CT_{ell}(G)}|W(G,T)|^{-1}\int_{T(F)}f(t)dt$$
  for every reasonable function $f$ on it. We use $D^G(\cdot)$ to denote the Weyl discriminant.

  In this case, $G = G_{2n}(F)$ and $H = H_{2n}(F)$. The geometric multiplicity is defined to be 
  $$m_{geom}(\theta)=\int_{\Gamma_{ell}(H)}D^H(t) \theta(t)dt$$
  where $\theta$ is a quasi-character on $ G_{2n}(F)$. The multiplicity formula
  $$m(\pi)=m_{geom}(\theta_\pi)$$
  has been proved for all supercuspidal representation of $G(F)$ in \cite{Beuzart-Plessis-Wan}. We can reduce the verification of \eqref{eq::Conjecture-EP} to the case where $\pi$ is a discrete series representation by the following two results.
  
  Let $P=MN$ be a maximal parabolic subgroup of $G$ with $M = G_{n_1}\times G_{n_2}$ ($2n=n_1+n_2$) and assume that $\theta$ is the parabolic induction of a quasi-chracter $\theta_M=\theta_1\otimes \theta_2$ of $M(F)$ (in the sense of Section 4.7 of \cite{Beuzart-Plessis-UnitaryGGP-Archimedean}). If $n_1$ is odd, we define
  $$m_{geom,M}(\theta_M)=0.$$
  If $n_1$ is even, then $M$ has a subgroup $H_M = H_{n_1} \times H_{n_2}$. We define
  $$m_{geom,M}(\theta_M)=\int_{\Gamma_{ell}(H_M)}D^{H_M}(t) \theta_M(t)dt.$$

  \begin{prop}
  	With the notation above, we have
  	$$m_{geom}(\theta)=m_{geom}(\theta_M).$$
  \end{prop}

  \begin{proof}
  	If $n_1$ is odd, by Proposition 4.7.1 of \cite{Beuzart-Plessis-UnitaryGGP-Archimedean} and the description of $\Gamma_{ell}(H)$ in \cite{Waldspurger-TransferFactor-ClassicalGroup-Manuscripta}, we know that $\theta(t)=0$ for all $t\in \Gamma_{ell}(H)$. This proves the proposition.
  	
  	If $n_1$ is even, by the description of $\Gamma_{ell}(H)$ in \cite{Waldspurger-TransferFactor-ClassicalGroup-Manuscripta}, we have a natural measure-preserving map $\iota:\Gamma_{ell}(H_M)\rightarrow \Gamma_{ell}(H)$. Using Proposition 4.7.1 of \cite{Beuzart-Plessis-UnitaryGGP-Archimedean} again, we know that (note that $(\frac{D^H(t)}{D^{H_M}(t')})^2=\frac{D^G(t)}{D^M(t')}$)
  	$$D^H(t)\theta(t)=\sum_{t'}D^{H_M}(t')\theta_M(t')$$
  	where $t'$ runs over elements in $\Gamma_{ell}(H_M)$ such that $\iota(t')=t$ (this is always a finite set, if it is empty then the right hand side is just 0). This proves the proposition.
  \end{proof}
   
  \begin{prop}
  	 Let $\pi_1$ and $\pi_2$ be smooth representations of finite length of $G_{n_1}$ and $G_{n_2}$ respectively ($2n = n_1 + n_2$). If $n_1$ (hence $n_2$) is odd, then
  	 \begin{align*}
  	 	\EP_H (\pi_1 \times \pi_2, \BC) = 0.
  	 \end{align*}
  	 If $n_1$ (hence $n_2$) is even, then
  	 \begin{align*}
  	 	\EP_H(\pi_1 \times \pi_2, \BC) = \EP_{H_{n_1}} (\pi_1, \BC) \cdot \EP_{H_{n_2}} (\pi_2, \BC).
  	 \end{align*}
  \end{prop}
 \begin{proof}
 	Let $P = MU$ for the parabolic subgroup of type $(n_1,n_2)$. If $n_1$ is odd, by part (2) of Lemma \ref{lem::cor-geo-lemm-Sp}, no $P$-orbit is relevant to $\EP_{H} (\pi_1 \times \pi_2, \BC)$. If $n_1$ is even, by the same lemma, $\EP_{H} (\pi_1 \times \pi_2, \BC)$ is supported only on one $M$-admissible orbit. The proposition then follows from the multiplicity of $\EP$ which is a direct consequence of the Kunneth theorem \ref{thm::Kunneth}. 
 \end{proof}

  \subsection{Relatively supercuspidal representations}
 
 We begin with some general notation. For a connected reductive group $\BG$ defined over $F$ and $G = \BG(F)$. Let $\BH \subset \BG$ be a closed subgroup and $H = \BH (F)$. Let $A_{G,H}$ be the maximal split torus of $Z(G) \cap H$ with $Z(G)$ the center of $G$. Let $\pi$ be an irreducible representation of $G$ with central character $\omega_{\pi}$. We say that $\pi$ is $H$-distinguished if $\Hom_H (\pi, \BC) \neq 0$. In particular, the central character $\omega_{\pi}$ of $\pi$ is trivial on $A_{G,H}$. We say that $\pi$ is $H$-relatively supercuspidal if $\pi$ is $H$-distinguished and for every nonzero $\ell \in \Hom_H (\pi, \BC) $, the relative matrix coefficient $\ell (\pi (\cdot )v)$ for every $v$ in the space of $\pi$ is compactly supported modulo $Z_G H$. Assume that the central character $\omega_{\pi}$ is unitary. We say that $\pi$ is a $H$-relatively discreter series representation if $\pi$ is $H$-distinguished and for all nonzero $\ell \in \Hom_H (\pi, \BC)$ and $v$ in the space of $\pi$, the relative matrix coefficient $\ell (\pi (\cdot) v)$ lies in $L^2 (Z_G H \backslash G)$. It is clear that an $H$-relatively supercuspidal representation is an $H$-relatively discrete series representation.
 
 Now we focus on our symplectic case. It is shown in \cite[Proposition 8.3.4]{Kato-Takano-Subrepresentation-symmetric} by the Jacquet module method that $Z([\nu^{-1/2}\rho, \nu^{1/2}\rho])$ is $H$-relatively supercuspidal for any irreducible cuspidal representation $\rho$ of $G_d$. Following the idea in \cite{Cai-Fan-RelativeCuspidal-Ext} we reprove this result and prove also that they actually exhaust all $H$-relative supercuspidal representations.
 
 \begin{prop}\label{prop::RSC-symplectic}
 	An irreducible representation $\pi$ of $G$ is $H$-relatively supercuspidal if and only if $\pi = Z([\nu^{-1/2}\rho, \nu^{1/2}\rho])$ for some irreducible cuspidal $\rho$. 
 \end{prop}
 
 \begin{proof}
 	By \cite[Corollary, Section 1.2]{Lapid-Offen-Explicit-Plancherel-GL-Sp}, an irreducible representation $\pi$ of $G$ is an $H$-relatively discrete series if and only if $\pi = \textup{L}(\nu^{-1/2}\Delta, \nu^{1/2}\Delta)$ for some segment $\Delta$. So we need to show that such a representation $\pi$ is $H$-relatively supercuspidal if and only if $l(\Delta) = 1$. Now let $\pi = \textup{L}(\nu^{-1/2}\Delta, \nu^{1/2}\Delta)$. It is well known that the Aubert-Zelevinsky involution $D(\pi)$ of $\pi$  is $Z(\nu^{-1/2}\Delta, \nu^{1/2}\Delta)$ (see \cite[Appendix A.5]{Lapid-Minguez-ParabolicInduction}). Since we have multiplicity one in this case, by \cite[Proposition 4.2]{Cai-Fan-RelativeCuspidal-Ext}, $\pi$ is $H$-relatively supercuspidal if and only if 
 	\begin{align*}
 		\Ext{2l-1}_H (  D(\pi)^{\vee}, \BC) = \Ext{2l-1}_H (Z (\nu^{-1/2}\Delta^{\vee}, \nu^{1/2}\Delta^{\vee}), \BC) \neq 0,
 	\end{align*}
 	where $l = l (\Delta)$ and $\Delta^{\vee}$ is the dual segment of $\Delta$. When $l = 1$ so that $\Delta = \{ \rho\}$, $D(\pi)$ is nothing but the generalized Steinberg representation $\St_2(\rho)$. The if part of the proposition then follows from Theorem \ref{thm::Ext-DS-GL-Sp}. For the only if part, by Proposition \ref{prop::Zel-ShortExactSeq}, we first have an exact sequence
 	\begin{align}\label{eq::RSC-short}
 		0 \ra Z(\nu^{-1/2}\Delta^{\vee}, \nu^{1/2}\Delta^{\vee}) \ra Z(\nu^{1/2}\Delta^{\vee}) \times Z(\nu^{-1/2} \Delta^{\vee}) \ra   Z(\Delta_1) \times Z(\Delta_2) \ra 0,
 	\end{align}
 	where $\Delta_1$ resp. $\Delta_2$ is the union resp. intersection of the segments $\nu^{-1/2}\Delta^{\vee}$ and $\nu^{1/2}\Delta^{\vee}$. By the long exact sequence associated to \eqref{eq::RSC-short}, it suffices for us to show that, for $l > 1$, 
 	\begin{align}\label{eq::RSC-symplectic-I}
 		\Ext{i}_H ( Z(\nu^{1/2}\Delta^{\vee}) \times Z(\nu^{-1/2} \Delta^{\vee}),\BC) = 0 \text{ for all } i \geqslant 2l-1
 	\end{align}
 	and
 	\begin{align}\label{eq::RSC-symplectic-II}
 		\Ext{i}_H ( Z(\Delta_1)  \times Z(\Delta_2) ,\BC ) = 0 \text{ for all } i \geqslant 2l.
 	\end{align}
 	The proofs of \eqref{eq::RSC-symplectic-I} and \eqref{eq::RSC-symplectic-II} are similar, so we prove only the former one. Let $P = MU$ be the standard parabolic subgroup of type $(ld,ld)$. For any $P$-orbit $\CO$, let $L,x$ and $\delta_x$ be associated to $\CO$ as in Section \ref{sec::gemeotric lemma}. Note that $L_x$ is a direct product of symplectic groups and general linear groups. By Proposition \ref{prop::Geometric-lemma}, the Kunneth theorem \ref{thm::Kunneth} and the description of $L_x,\delta_x$ in Lemma \ref{lem::admissible orbit-Sp}, we see that 
 	\begin{align*}
 		\Ext{2l-1}_H ( (Z(\nu^{1/2}\Delta^{\vee}) \times Z(\nu^{-1/2} \Delta^{\vee}))_{\CO},\BC)
 	\end{align*}
 	is a direct sum of tensor products of extension groups of the form
 	\begin{align*}
 		   \Ext{i}_{H_{2n}} (Z(\Delta_2),\BC)  \text{ and }\Ext{j}_{G_{m}} (Z(\Delta_3) ,Z(\Delta_3)).
 	\end{align*}
 	By \cite[Proposition 2.9]{Prasad-Ext-BranchingLaw-ICM2018}, we have
 	\begin{align*}
 		\Ext{j}_{G_m} (Z(\Delta_3),Z(\Delta_3)) = 0 \text{ for all } j > l(\Delta_3).
 	\end{align*}
 	Thus, combined with Proposition \ref{prop::speh-ext}, the largest number $d$ such that 
 	\begin{align*}
 		\Ext{d}_H ( (Z(\nu^{1/2}\Delta^{\vee}) \times Z(\nu^{-1/2} \Delta^{\vee}))_{\CO},\BC)
 	\end{align*}
 	is nonzero is not greater than $l$, which is strictly less than $2l-1$ as $l >1$. This finishes the proof of the proposition.
 \end{proof}

 At the end of this section we raise a question on the branching laws from $G$ to $H$. It is known that the restriction to $H$ of any supercuspidal representation of $G$ is projective. Note that this is not true for $H$-relatively supercuspidals. For example, when $n = 1$, the trivial representation is $H$-relatively supercuspidal, but by \cite[Theorem 1]{Orlik-Extensions-Steinberg-JoA},
 \begin{align*}
 	\Ext{1}_{\SL_2(F)} (\BC,\St) = \BC.
 \end{align*}
 From this perspective, one might pose the following conjecture.

 \begin{conj}\label{conj::branching}
 		Let $\pi$ be an irreducible representation of $G$. Then the restriction $\pi|_H$ to $H$ of $\pi$ is projective if and only if $\pi$ is supercuspidal.
 	
 \end{conj}
 
 Note that a necessity for Conjecture \ref{conj::branching} to be true is the fact that, for any smooth representation $\pi$ of $\GL_n(F)$, its restriction to $\SL_n(F)$ is projective if and only if $\pi$ is supercuspidal. This can be easily shown by the result of Adler and Roche in \cite{Adler-Roche-Proj-Cuspidal}.

 \section{The linear pair}\label{sec::linear}

 Let $G = G_{2n}$ and $H \cong G_n \times G_n $ be the subgroup of $G_{2n}$ of matrices of the form
 \begin{align*}
 	H=  \begin{pmatrix}
 		g_1 & \\
 		  & g_2
 	\end{pmatrix} \text{ with } g_1,g_2 \in G_n.
 \end{align*}
 Let $Z = Z_{2n}$ be the center of $G_{2n}$.
 
 \subsection{The orbit analysis}
 
 Let 
 \begin{align*}
 	\varepsilon = \begin{pmatrix}
 		I_n & \\
 		  & - I_n
 	\end{pmatrix}
 \end{align*}
 and $\theta$ be the involution on $G$ defined by $\theta(g)  = \varepsilon g \varepsilon^{-1}$. So $H = G^{\theta}$. Let
 \begin{align*}
 	X = \{x \in G \colon \theta(x) = x^{-1}\}
 \end{align*}
 equipped with the $G$-action given by $g \cdot x = g x \theta(g)^{-1}$. Note that the map $x \mapsto x \varepsilon$ givens a bijection of $X$ onto the subset of elements of order $2$ in $G$ and that the action of $G$ on $X$ is transformed to the conjugation action of $G$. 
 
 For our purposes, it suffices to consider parabolic orbits in $G \cdot e$ when the parabolic subgroup is maximal. Let $P = P_{(m,2n-m)} = MU$ be the standard parabolic subgroup of $G$ of type $(m,2n-m)$. The explicit description of $P$-orbits in $G \cdot e$ has been given in \cite[Section 4]{YangChang-LinearPeriods-MathZ}. We recall the results there and refer the reader to \cite{YangChang-LinearPeriods-MathZ} for the proofs of the following two lemmas.
 
 Let 
 \begin{align*}
 	\mathcal{I}_m  =\{ (r,s) \in \BN \times \BN  \colon r +s \leqslant m, \ m- r\leqslant n \text{ and } m -s \leqslant n \}.
 \end{align*}
	For $(r,s) \in \CI_m$, let
	\begin{align*}
		w_{(r,s)} = \begin{pmatrix}
			I_{ r + s}  & & &  \\
			&  &  I_{m - r -s}  &  \\
			& I_{m -r -s} & &  \\
			&  &  &  I_{2n-2m + r + s }
		\end{pmatrix}.
	\end{align*}
 We have that $M \cap w_{(r,s)} M w_{(r,s)}^{-1}$ is the standard Levi subgroup of type $(r+s, m-r-s,m-r-s,2n-2m+r+s)$. 
\begin{lem}\label{lem::linear-fibration}
	There is a bijection between $\CI_m$ and the set of $P$-orbits in $G \cdot e$. For $(r,s) \in \CI_m$, let $\CO_{(r,s)}$ be the orbit corresponding to $(r,s)$. Then we have $\iota_M (\CO_{(r,s)})   =  w_{(r,s)}$, where
	\begin{align*}
		\iota_M \colon P \backslash \,G \cdot e  &\ra  {}_MW_M \cap \mathfrak{I}_0(\theta)    
	\end{align*}
	is the fibration map in Section \ref{sec::gemeotric lemma}.
\end{lem}

\begin{lem}\label{lem::linear-general-orbit}
	For $(r,s) \in \CI_m$, let $\CO$ be the $P$-orbit corresponding to $(r,s)$. Let $w = w_{(r,s)}$ and $Q $ be the standard parabolic subgroup with $L = M \cap wMw^{-1}$. Then there is a good representative $x \in \CO \cap L w$ such that
	\begin{itemize}
		\item[(1)] $L_x = \{ \textup{diag}(g_1, g_2,g,g,g_3,g_4) \colon g_1 \in G_r,g_2 \in G_s, g\in G_{m -r -s},g_3 \in G_{n+s -m}, g_4 \in G_{n+r -m}\},$ 
		\item [(2)] $(\delta_{Q_x} \delta_Q^{-1/2}) (\textup{diag} (g_1,g_2,g,g,g_3,g_4))  = \nu (g_1)^{(s-r)/2}  \nu (g_2)^{(r-s)/2}  \nu (g_3)^{(s-r)/2}  \nu (g_4)^{(r-s)/2}.$
	\end{itemize} 
\end{lem}

 \subsection{Consequences of the geometric lemma}

\begin{lem}\label{lem::cor-geo-lemm-linear}
	Let $P = MU$ be the standard parabolic subgroup of $G$ of type $(m,2n-m)$ and $\sigma$ a finite length smooth representation of $M$ on which $Z$ acts trivially. Let $\pi = I_P^G (\sigma)$. For $(r,s) \in \CI_m$, let $\CO$ be the $P$-orbit corresponding to $(r,s)$. Then:
	
	\textup{(1)}$\quad$ If $r + s = m $, then $L = M$ and $\CO$ is $M$-admissible. In this case, if $\CO$ is relevant to $\Ext{\sharp}_{H/Z} (\pi,\BC)$, then there exists a composition factor $\rho_1 \otimes \rho_2$ of $\sigma$ such that
	\begin{align*}
		w_{\rho_1} = |\cdot|_F^{-(r-s)^2/2}  \text{ and } w_{\rho_2} = |\cdot |_F^{(r-s)^2/2};
	\end{align*}
	If $r + s < m$ and $\CO$ is relevant to $\Ext{\sharp}_{H/Z} (\pi,\BC)$, then there exists a composition factor $\rho_1 \otimes \rho_2 \otimes \rho_3 \otimes \rho_4$ of $r_{L,M} (\sigma)$ such that
	\begin{align*}
		w_{\rho_1} = |\cdot|_F^{-(r-s)^2/2},\ w_{\rho_2} = w_{\rho_3}^{-1}  \text{ and } w_{\rho_4} = |\cdot |_F^{(r-s)^2/2}.
	\end{align*}
		
	\textup{(2)}$\quad $ If $\CO$ is relevant to $\EP_{H/Z} (\pi,\BC)$, we have $m = n$ and $ r = s =0$. 
\end{lem}
 \begin{proof}
 	For (1), we take $z = \textup{diag}(tI_{r+s}, I_{2n-r-s})$,viewed as a central element in $L_x /Z$ and apply Lemme \ref{lem::vanishing}. We can prove the assertion for $w_{\rho_1}$. Similarly for other central characters. For (2), this follows from Proposition \ref{prop::EP-trivial} since oterwise we have a nontrivial central subgroup of $L_x/Z$.
 \end{proof}

 \subsection{Vanishing of higher extension groups}
 
 Let $\rho$ be an irreducible cuspidal representation of $G_d$ with trivial central character. Let 
 \begin{align*}
 	\Delta = \{ \nu^{-(k-1)/2}\rho, \nu^{-(k-3)/2}\rho,\cdots,\nu^{(k-1)/2} \rho \}
 \end{align*}
 be a segment and $\pi = \textup{L}(\Delta)$ be the essentially square-integrable representations associated to $\Delta$.
 
 \begin{lem}\label{lem::linear-cuspidal}
 	Let $\rho, d$ be as above. Let $r,s$ be two nonnegative integers such that $d = r + s$. We have
 	\begin{align*}
 		&\Ext{1}_{G_r \times G_s} (\rho ,\BC) \cong  \Hom_{G_r \times G_s}  (\rho, \BC) \\
 		&\Ext{i}_{G_r \times G_s} (\rho ,\BC)  = 0, \ \text{ for all } i \geqslant 2.
 	\end{align*}
 \end{lem}
 \begin{proof}
 	Let $Z$ be the center of $G_d$. By Lemma \ref{lem::HS-spectral}, we have a spectral sequence
 	\begin{align*}
 		E^{p,q}_2 = \Ext{q}_{G_r \times G_s /Z} (H_p (Z,\rho), \BC)  \Longrightarrow  \Ext{p+q}_{G_r \times G_s} (\rho ,\BC)
 	\end{align*}
 	By Lemma \ref{lem::homology-G1G2} and Lemma \ref{lem::Ext-torus}, we have 
 	\begin{align*}
 		H_0(Z,\rho) = H_1 (Z,\rho) = \rho, \quad H_i (Z,\rho) = 0, \ i \geqslant 2.
 	\end{align*}
 	Thus the spectral sequence collapse at page $E_2$. Since $\rho$ is projective in $\CM(G_d/Z)$, its restriction to $G_r \times G_s$ is also  projective in $\CM (G_r \times G_s/Z)$. The lemma then follows easily from the spectral sequence above.
 \end{proof}

\begin{thm}\label{thm::Vanishing-linear}
	We have $\Ext{i}_{H/Z} (\pi,\BC) = 0$ for $i \geqslant 1$.
\end{thm}
 \begin{proof}
 	We are going to prove this by induction on $k$. The case $k  = 1$ is obvious since the restriction of $\rho$ to $H$ is projective in $\CM(H/Z)$. Assume that $k \geqslant 2$. Let $\Delta_1 = \Delta \backslash \{\nu^{-(k-1)/2} \rho \}$. Then by Proposition \ref{prop::Zel-ShortExactSeq} we have two short exact sequences
 	\begin{align}\label{formula::linear-exact-I}
 		0\ra \pi \ra \textup{L}(\Delta_1) \times \nu^{-(k-1)/2} \rho \ra L(\nu^{-(k-1)/2} \rho, \Delta_1) \ra 0
 	\end{align}
 	and
 	\begin{align}\label{formula::linear-exact-II}
 		0 \ra  L(\nu^{-(k-1)/2} \rho, \Delta_1)  \ra \nu^{-(k-1)/2} \rho  \times \textup{L}(\Delta_1) \ra \pi \ra 0
 	\end{align}
 	By the long exact sequences associated to \eqref{formula::linear-exact-I} and \eqref{formula::linear-exact-II}, the theorem will follow from the fact that
 	\begin{align}\label{formula::linear-ext-0-I}
 		\Ext{i}_{H/Z} ( \textup{L}(\Delta_1) \times \nu^{-(k-1)/2} \rho, \BC) = 0, \ \text{ for all } i \geqslant 1
 	\end{align}
 	and that
 	\begin{align}\label{formula::linear-ext-0-II}
 			\Ext{i}_{H/Z} (\nu^{-(k-1)/2} \rho  \times \textup{L}(\Delta_1), \BC) =  0 ,\ \text{ for all } i \geqslant 2.
 	\end{align}
 	
	For \eqref{formula::linear-ext-0-I}, when $k = 2$, the only possible $P$-orbit $\CO$ that is relevant to 
	\begin{align*}
		\Ext{i}_{H/Z} ( \nu^{1/2}\rho \times \nu^{-1/2} \rho, \BC)
	\end{align*}
	is such that $\iota_M (\CO) = w_M$. By Proposition \ref{prop::Geometric-lemma} and Lemma \ref{lem::linear-general-orbit}, we have
	\begin{align*}
		\Ext{i}_{H/Z} ( \nu^{1/2}\rho \times \nu^{-1/2} \rho, \BC)  \cong \Ext{i}_{G_d / Z_d} (\rho \otimes \rho, \BC) \cong \Ext{i}_{G_d/Z_d}  (\rho ,\rho^{\vee}).
	\end{align*}
	Thus the formula \eqref{formula::linear-ext-0-I} holds in this case since $\rho$ is projective in $\CM(G_d /Z_d)$. When $ k > 2$, the formula \eqref{formula::linear-ext-0-I} follows easily from the central character considerations in part (1) of Lemma \ref{lem::cor-geo-lemm-linear}. 
	
	Now we prove \eqref{formula::linear-ext-0-II}. It is easy to verify a special case where $d = 1$ and $k = 2$. That is, we have
	\begin{align*}
			\Ext{i}_{G_1 \times G_1/Z_2} (\nu^{-1/2}   \times \nu^{1/2}, \BC) =  0 ,\ \text{ for all } i \geqslant 2.
	\end{align*}
	In general, let $P = MU$ be the standard parabolic subgroup of $G_{2n}$ of type $(d,(k-1)d)$. For a $P$-orbit $\CO$, let $w,L$ and $x$ be the data associated to $\CO$. By the cuspidality of $\rho$, the left hand side of \eqref{formula::linear-ext-0-II} is supported only on those $\CO_{(r,s)}$ such that  $r = s = 0$ or $r + s = d$. We split the discussion into two cases.
	
	Case (I). Let $r + s = d$ and $\CO = \CO_{(r,s)}$. If $-(k-1)/2 \neq -(r-s)^2/2$, then $\CO$ is not relevant by the central character considerations in part (1) of Lemma \ref{lem::cor-geo-lemm-linear}. Assume on the contrary that $(k-1)/2 = (r-s)^2/2$. We may further assume that $d > 1$. Then by Proposition \ref{prop::Geometric-lemma} we have
	\begin{align*}
		\Ext{i}_{H/Z} ((\nu^{-(k-1)/2} \rho  \times \textup{L}(\Delta_1))_{\CO}, \BC)  =  \Ext{i}_{L_x/Z} ( \rho  \otimes \textup{L}(\Delta_1)', \BC),
	\end{align*}
	where $\textup{L}(\Delta_1)'$ is a twist of $\textup{L}(\Delta_1)$ with a trivial central character. Note that we may view $G_r \times G_s$ as a normal subgroup of $L_x/Z$ with quotient isomorphic to $G_{n-r} \times G_{n-s} /Z_{2n-r-s}$. Thus by Lemma \ref{lem::HS-spectral} and Lemma \ref{lem::homology-G1G2}, we have a spectral sequence
	\begin{align}\label{formula::linear-main-spectral}
		\Ext{q}_{G_{n-r} \times G_{n-s} /Z_{2n-r-s}}  (H_p (G_r \times G_s, \rho) \otimes \textup{L}(\Delta_1)', \BC) \Longrightarrow \Ext{p+q}_{L_x/Z} ( \rho  \otimes \textup{L}(\Delta_1)', \BC).
	\end{align} 
	Note that we have $d > 1$ and $r \neq s$ (otherwise $k =1$). By \cite[Theorem 3.1]{Matringe-Linear+Shalika-JNT}, we have
	\begin{align*}
		\Hom_{G_r \times G_s} (\rho, \BC) = 0.
	\end{align*}
	Thus by Lemma \ref{lem::linear-cuspidal} and \eqref{formula::dual-homology-ext}, we have $H_p (G_r \times G_s, \rho )  = 0$ for all $p$. So by \eqref{formula::linear-main-spectral}, we have that $\CO$ is not relevant to the left hand side of \eqref{formula::linear-ext-0-II}. 
	
	Case (II). Let $r = s = 0$. Let $\CO = \CO_{(0,0)}$. By Proposition \ref{prop::Geometric-lemma}, we have
	\begin{align*}
		\Ext{i}_{H/Z} ((\nu^{-(k-1)/2} \rho  \times \textup{L}(\Delta_1))_{\CO}, \BC)  =  \Ext{i}_{L_x/Z} ( \nu^{-(k-1)/2}\rho \otimes \nu^{(k-1)/2} \rho \otimes \textup{L}(\Delta_2), \BC),
	\end{align*}
	where $\Delta_2 = \Delta \backslash \{ \nu^{-(k-1)/2} \rho, \nu^{(k-1)/2} \rho \}$. In this case we view $G_d$ as a normal subgroup of $L_x/Z$ with quotient isomorphic to $G_{n-d,n-d}/Z_{2n-2d}$. Thus
	\begin{equation}\label{formula::linear-main-spectral-II}
		\begin{aligned}
			\Ext{q}_{G_{n-d} \times G_{n-d} /Z_{2n-2d}}  (H_p (G_d, &\nu^{-(k-1)/2}\rho \otimes \nu^{(k-1)/2} \rho) \otimes \textup{L}(\Delta_2), \BC) \\
			&\Longrightarrow \Ext{i}_{L_x/Z} ( \nu^{-(k-1)/2}\rho \otimes \nu^{(k-1)/2} \rho \otimes \textup{L}(\Delta_2), \BC).
		\end{aligned}
	\end{equation}
	Note that by Lemma \ref{lem::tensor-hom-ext},
	\begin{align*}
		\Ext{p}_{G_d} (\nu^{-(k-1)/2}\rho \otimes \nu^{(k-1)/2} \rho, \BC) \cong \Ext{p}_{G_d} (\rho, \rho^{\vee}),\text{ for all }p \geqslant 0.
	\end{align*}
	So if $\rho$ is not self-dual, we have that $\CO$ is not relevant to the lefh hand side of \eqref{formula::linear-ext-0-II}. This finishes the proof of  \eqref{formula::linear-ext-0-II}, actually for all $i \geqslant 0$. If $\rho$ is self-dual, then 
	\begin{align*}
		\Ext{p}_{G_d} (\rho,\rho^{\vee}) = \begin{cases*}
			 \BC,\quad p =0,1,\\
			 0,\quad \text{oterwise}.
		\end{cases*}
 	\end{align*}
	Hence the spectral sequence \eqref{formula::linear-main-spectral-II} collapse at page $E_2$ and \eqref{formula::linear-ext-0-II} follows easily from the induction hypothesis and the convergence of the spectral sequence \eqref{formula::linear-main-spectral-II}.
 \end{proof}
 
 \begin{rem}
 	It is not true that the higher extension group vanishes for all tempered representations. In fact, let $\pi_1$ resp. $\pi_2$ be irreducible cuspidal representations of $G_{2n_1}$ resp. $G_{2n_2}$, $n = n_1 + n_2$, such that they are $G_{n_1} \times G_{n_1}$ resp.  $G_{n_2} \times G_{n_2}$-distingusihed. Then one may see easily that $\Ext{1}_{H/Z} (\pi_1 \times \pi_2, \BC) \cong \BC$. 
 \end{rem}

 \subsection{The multiplicity formula}\label{sec::EP-linear}

 Let $H_0=\{\textup{diag}(a,a)|\;a\in G_n\}$. In this case the geometric multiplicity is defined to be 
 $$m_{geom}(\theta)=\int_{\Gamma_{ell}(H_0/Z_G)}D^H(t) c_\theta(t)dt$$
 where $\theta$ is a quasi-character on $G/Z_G$. Since for discrete series the multiplicity is equal to the multiplicity of the Shalika model, combining with the multiplicity formula of the Shalika model proved in \cite{Beuzart-Plessis-Wan-Duke-Shalika}, we know that the multiplicity formula
 $$m(\pi)=m_{geom}(\theta_\pi)$$
 holds for all discrete series of $G/Z_G$. 
 
 Let $P=MN$ be a parabolic subgroup of $G$ and assume that $\theta$ is the parabolic induction of a quasi-chracter $\theta_M$ of $M$ (in the sense of Section 4.7 of \cite{Beuzart-Plessis-UnitaryGGP-Archimedean}). The next proposition is a direct consequence of Proposition 4.7.1 of \cite{Beuzart-Plessis-UnitaryGGP-Archimedean}.
 
 \begin{prop}\label{prop::induction-geom-linear}
 	If $M\neq G_n\times G_n$, then $m_{geom}(\theta)=0$. If $M = G_n\times G_n$, we may just take $M$ to be $H$. Then 
 	$$m_{geom}(\theta)=\int_{\Gamma_{ell}(H_0/Z_G)} D^{H_0}(t)\theta_M(t)dt.$$
 \end{prop}
 \begin{prop}\label{prop::induction-EP-linear}
 	Let $2n = n_1 + n_2$. Let $\pi_1$ resp. $\pi_2$ be a finite length smooth representation of $G_{n_1}$ resp. $G_{n_2}$ with central characters $\chi$ resp. $\chi^{-1}$ so that $\pi_1 \times \pi_2$ has trivial central character. If $n_1 \neq n_2$, then we have 
 	\begin{align*}
 		\EP_{H/Z} (\pi_1 \times \pi_2,\BC)  =  0.
 	\end{align*}
 	If $n_1 = n_2 = n$, then
 	\begin{align*}
 		\EP_{H/Z} (\pi_1 \times \pi_2,\BC) = \EP_{G_n/Z_n}(\pi_1 \otimes \pi_2,\BC). 
 	\end{align*}
 \end{prop}
 \begin{proof}
 	This follows immediately from part (2) of Lemma \ref{lem::cor-geo-lemm-linear} and Proposition \ref{prop::Geometric-lemma}.
 \end{proof}
 \begin{cor}
 	We have $\EP_{H/Z} (\pi,\BC) = m_{geom}(\pi)$ for any finite length representation $\pi$ of $G$.
 \end{cor}
 \begin{proof}
 	By Proposition \ref{prop::induction-geom-linear} and Proposition \ref{prop::induction-EP-linear} it suffices to show that, when $M = H$ and $\pi = \pi_1 \times \pi_2$, we have
 	\begin{align*}
 		\EP_{G_n/Z_n}(\pi_1 \otimes \pi_2,\BC)  = \int_{\Gamma_{ell}(H_0/Z_G)} D^{H_0}(t)\theta_M(t)dt.
 	\end{align*}
 	This follows immediately from the Kazhdan orthogonality proved by Schneider and Stuhler and by Bezrukavnikov (see \cite[Theorem 6.1]{Prasad-IHES-Note}).
 \end{proof}

 \subsection{$H$-relatively supercuspidal representations}
  For our last result it will be convenient to take  $\theta = \text{Ad}(\varepsilon)$ and $H = G^{\theta}$, where $\varepsilon = \textup{diag}(1,-1,1,\cdots,1,-1)$. Let $M$ be a standard Levi subgroup of $G$ of type $(n_1,\cdots,n_k)$. Recall that an irreducible representation $\sigma_1 \boxtimes \sigma_2 \boxtimes \cdots \boxtimes \sigma_k$ of $M$ is said to be regular if $\sigma_i \ncong \sigma_j$ for $i \neq j$. 
 \begin{prop}\label{prop::RSC-linear}
 	An irreducible representation $\pi$ of $G$ is $H$-relatively supercuspidal if and only if $\pi = I_P^G (\sigma)$ where $P = MU$ is the standard parabolic subgroup of type $(2n_1,2n_2,\cdots,2n_k)$ and $\sigma$ is a regular $M \cap H$-distinguished supercuspidal representation of $M$.
 \end{prop}
\begin{proof}
	The line of arguments is very similar to that in the Jacquet-Rallis case in \cite{Cai-Fan-RelativeCuspidal-Ext}, so we omit the details. We point out only that the Plancherel decomposition for $G/H$ has been established in the work of Duhamel \cite{Duhamel-Plancherel-Linear-Shalika}. It follows that $\pi$ is an $H$-relatively discrete series representation if and only if $\pi$ lies in the image of discrete series of $\SO_{2n+1}(F)$ under the endoscopy transfer. By \cite{Arthur-Endoscopic}, this is equivalent to say that $\pi = I_P^G (\sigma)$ where $\sigma$ is a regular $M \cap H$-distinguished discrete series of $M$.  
\end{proof}
 

%

\end{document}